\documentclass[12pt]{amsart}
\usepackage[margin=1in]{geometry}
\usepackage{amssymb,amsfonts,amsmath,mathtools}
\usepackage{color}
\usepackage{soul}
\usepackage{enumerate}
\usepackage{mathrsfs}
\usepackage[unicode]{hyperref}
\usepackage[capitalise]{cleveref}
\usepackage[normalem]{ulem}

\usepackage{constants}
\usepackage{bbm}
\newtheorem{theorem}{Theorem}[section]
\newtheorem{corollary}[theorem]{Corollary}

\newtheorem{proposition}[theorem]{Proposition}
\newtheorem{lemma}[theorem]{Lemma}

\newtheorem{question*}{Question}
\newtheorem{problem*}{Problem}

\theoremstyle{definition}

\theoremstyle{remark}
\newtheorem{remark}[theorem]{Remark}

\numberwithin{equation}{section}

\crefname{figure}{Figure}{Figures}
\theoremstyle{plain}
\newtheorem*{theorem*}{Theorem}
\crefname{theorems}{Theorem}{Theorems}
\crefname{corollaries}{Corollary}{Corollaries}
\newtheorem*{corollary*}{Corollary}
\crefname{corollaries*}{Corollary}{Corollaries}
\crefname{lemma}{Lemma}{Lemmata}
\crefname{proposition}{Proposition}{Propositions}
\crefname{conjectures}{Conjecture}{Conjectures}
\newtheorem*{conjonjecture*}{Conjecture}
\crefname{conjonjectures*}{Conjecture}{Conjectures}
\crefname{definitions}{Definition}{Definitions}
\crefname{hypotheses}{Hypothesis}{Hypotheses}

\renewcommand{\hat}{\widehat}

\newcommand{\Z}{\mathbb{Z}}
\newcommand{\R}{\mathbb{R}}
\newcommand{\Q}{\mathbb{Q}}

\newcommand{\kn}{\mathfrak{n}}

\newcommand{\kp}{\mathfrak{p}}

\newcommand{\re}{\textup{Re}}
\newcommand{\im}{\textup{Im}}

\newcommand{\Ad}{\mathrm{Ad}}

\newcommand{\GL}{\mathrm{GL}}

\newcommand{\SL}{\mathrm{SL}}

\newcommand{\N}{\mathrm{N}}
\newcommand{\A}{\mathbb{A}}

\makeatletter
\DeclareFontFamily{U}  {MnSymbolF}{}
\DeclareSymbolFont{symbolsMN}{U}{MnSymbolF}{m}{n}
\SetSymbolFont{symbolsMN}{bold}{U}{MnSymbolF}{b}{n}
\DeclareFontShape{U}{MnSymbolF}{m}{n}{
    <-6>  MnSymbolF5
   <6-7>  MnSymbolF6
   <7-8>  MnSymbolF7
   <8-9>  MnSymbolF8
   <9-10> MnSymbolF9
  <10-12> MnSymbolF10
  <12->   MnSymbolF12}{}
\DeclareFontShape{U}{MnSymbolF}{b}{n}{
    <-6>  MnSymbolF-Bold5
   <6-7>  MnSymbolF-Bold6
   <7-8>  MnSymbolF-Bold7
   <8-9>  MnSymbolF-Bold8
   <9-10> MnSymbolF-Bold9
  <10-12> MnSymbolF-Bold10
  <12->   MnSymbolF-Bold12}{}
\DeclareMathSymbol{\tbigtimes}{\mathop}{symbolsMN}{2}
\newcommand*{\bigtimes}{%
  \DOTSB
  \tbigtimes
  \slimits@ 
}
\makeatother

\renewcommand{\tilde}{\widetilde}

\renewcommand{\bar}{\overline}

\renewcommand{\epsilon}{\varepsilon}

\renewcommand{\pmod}[1]{\, (\mathrm{mod} {\, #1})}

\newcommand{\cO}{\mathcal{O}}

\newcommand{\kq}{\mathfrak{q}}

\renewcommand{\Re}{\mathrm{Re}}

\DeclareMathAlphabet{\mathpzc}{OT1}{pzc}{m}{it}

\renewcommand{\pmod}[1]{\,(\mathrm{mod}\,\,#1)}

\newconstantfamily{abcon}{symbol=c}

\makeatletter
\let\@wraptoccontribs\wraptoccontribs
\makeatother

\title[Zeros of Rankin--Selberg $L$-functions in families]{Zeros of Rankin--Selberg $L$-functions in families}
\author{Peter Humphries}
\address{Department of Mathematics, University of Virginia, Charlottesville, VA 22904, USA}
\email{\href{mailto:pclhumphries@gmail.com}{pclhumphries@gmail.com}}
\urladdr{\href{https://sites.google.com/view/peterhumphries/}{https://sites.google.com/view/peterhumphries/}}
\author{Jesse Thorner}
\address{Department of Mathematics, University of Illinois, Urbana, IL 61801, USA}
\email{\href{mailto:jesse.thorner@gmail.com}{jesse.thorner@gmail.com}}



\begin{document}

\begin{abstract}
Let $\mathfrak{F}_n$ be the set of all cuspidal automorphic representations $\pi$ of $\mathrm{GL}_n$ with unitary central character over a number field $F$.  We prove the first unconditional zero density estimate for the set $\mathcal{S}=\{L(s,\pi\times\pi')\colon\pi\in\mathfrak{F}_n\}$ of Rankin--Selberg $L$-functions, where $\pi'\in\mathfrak{F}_{n'}$ is fixed.  We use this density estimate to establish
\begin{enumerate}[(i)]
	\item a hybrid-aspect subconvexity bound at $s=\frac{1}{2}$ for almost all $L(s,\pi\times\pi')\in \mathcal{S}$,
	\item a strong on-average form of effective multiplicity one for almost all $\pi\in\mathfrak{F}_n$, and
	\item a positive level of distribution for $L(s,\pi\times\tilde{\pi})$, in the sense of Bombieri--Vinogradov, for each $\pi\in\mathfrak{F}_n$.
\end{enumerate}
\end{abstract}

\maketitle

\section{Introduction and statement of the main result}
\label{sec:intro}

Let $\A_F$ be the ring of ad\`{e}les over a number field $F$ with absolute norm $\N=\N_{F/\Q}$ and absolute discriminant $D_F$.  Let $\mathfrak{F}_{n}$ be the set of cuspidal automorphic representations $\pi=\bigotimes_{v} \pi_{v}$ of $\GL_{n}(\A_F)$, where the (restricted) tensor product runs over all places of $F$ and $\pi$ is normalized so that its central character is trivial on the diagonally embedded copy of the positive reals.  Let $\kq_{\pi}$ be the arithmetic conductor of $\pi$, $C(\pi)\geq 1$ the analytic conductor of $\pi$ (see \eqref{eqn:analytic_conductor_def}), and $\mathfrak{F}_n(Q)=\{\pi\in\mathfrak{F}_n\colon C(\pi)\leq Q\}$.  The analytic conductor $C(\pi)$ is a useful measure for the arithmetic and spectral complexity of $\pi$. Our normalization for the central characters ensures that $|\mathfrak{F}_n(Q)|$ is finite.

Given $\pi\in\mathfrak{F}_n$ and $\pi'\in\mathfrak{F}_{n'}$, let $L(s,\pi\times\pi')$ be the associated Rankin--Selberg $L$-function, and let $\tilde{\pi}\in\mathfrak{F}_n$ and $\tilde{\pi}'\in\mathfrak{F}_{n'}$ be the contragredient representations.  When $\pi'\in\{\tilde{\pi},\tilde{\pi}'\}$, work of Brumley \cite[Appendix]{Humphries} and the authors \cite{HT} shows that there exists an effectively computable constant $\Cl[abcon]{ZFR_standard_intro}=\Cr{ZFR_standard_intro}(n,n')>0$ such that $L(s,\pi\times\pi')$ has a ``standard'' zero-free region of the shape
\begin{equation}
\label{eqn:std_ZFR}
\re(s)\geq 1-\frac{\Cr{ZFR_standard_intro}}{\log(C(\pi)C(\pi')(|\im(s)|+3)^{[F:\Q]})}
\end{equation}
apart from at most one real simple zero.  This is comparable to the classical zero-free region for Dirichlet $L$-functions.  Brumley (\cite{Brumley} and \cite[Appendix]{Lapid}) established a much narrower zero-free region for all choices of $\pi$ and $\pi'$.  The generalized Riemann hypothesis (GRH) asserts that $L(s,\pi\times\pi')\neq 0$ for $\re(s)>\frac{1}{2}$.  Zeros near the line $\re(s)=1$ are typically most damaging in applications, but even a zero-free region of the shape $\re(s)\geq1-\delta$ for some constant $\delta=\delta(n,n',[F:\Q])>0$ would be sufficient to obtain many spectacular arithmetic consequences.

Since such strong zero-free regions for Rankin--Selberg $L$-functions remain out of reach, it is useful to show that zeros near the line $\re(s)=1$ must be ``sparse''.  A suitable quantitative formulation can serve as a proxy for a zero-free region of the shape $\re(s)\geq1-\delta$.  Famous consequences of this philosophy include Hoheisel's proof \cite{Hoheisel} that $p_{n+1}-p_n\ll p_n^{1-1/33000}$ (where $p_n$ is the $n$-th prime) and Linnik's proof \cite{Linnik} that if $\gcd(a,q)=1$, then there exists an absolute, effectively computable constant $B>0$ and a prime $p\leq q^B$ such that $p\equiv a\pmod{q}$.

To quantify our notion of ``sparse'', we define for $\sigma\geq 0$ and $T\geq 1$ the quantity
\[
N_{\pi\times\pi'}(\sigma,T)=|\{\rho=\beta+i\gamma\colon L(\rho,\pi\times\pi')=0,~\beta\geq\sigma,~|\gamma|\leq T\}|.
\]
Note that $N_{\pi\times\pi'}(\frac{1}{2},T)$ is roughly $T\log(C(\pi)C(\pi')T)$ via the argument principle and the functional equation, and GRH can be restated as $N_{\pi\times\pi'}(\sigma,T)=0$ for all $\sigma>\frac{1}{2}$.  The zero density estimate
\begin{equation}
\label{eqn:SoundThorner}
N_{\pi\times\pi'}(\sigma,T)\ll_{n,n',[F:\Q]}(C(\pi)C(\pi')T^{[F:\Q]})^{10^7 (n'n)^4(1-\sigma)}.
\end{equation}
follows from work of Soundararajan and Thorner \cite[Corollary 2.6]{ST}.  Therefore, while an arbitrary Rankin--Selberg $L$-function $L(s,\pi\times\pi')$ is not yet known to have the standard zero-free region \eqref{eqn:std_ZFR}, the bound \eqref{eqn:SoundThorner} ensures that the number of zeros in the region \eqref{eqn:std_ZFR} is $O_{n,n',[F:\Q]}(1)$.

Let $\mathcal{S}\subseteq\mathfrak{F}_n$, and let $\mathcal{S}(Q)=\{\pi\in\mathcal{S}\colon C(\pi)\leq Q\}$.  In this article, we seek a strong averaged form of \eqref{eqn:SoundThorner}, namely
\begin{equation}
\label{eqn:goal}
	\sum_{\pi\in\mathcal{S}(Q)}N_{\pi\times\pi'}(\sigma,T)\ll_{n,n',[F:\Q],\epsilon}(Q|\mathcal{S}(Q)|C(\pi')T^{[F:\Q]})^{A(1-\sigma)+\epsilon},
\end{equation}
where $A=A(n,n',[F:\Q])>0$ is a constant and $\epsilon>0$.  The bound \eqref{eqn:goal} follows from the works of Brumley, Thorner, and Zaman under at least one of the following hypotheses:\footnote{In \cite{BTZ,ST}, it is assumed that $F=\Q$.  Uniformity over $F\neq\Q$ requires minor modifications.}
\begin{itemize}
	\item $\pi'\in\mathfrak{F}_1$ is trivial \cite[Theorem 1.2]{TZ_GLn},
	\item $\max\{n,n'\}\leq 4$ \cite[Theorem 1.3]{BTZ}, or
	\item $\pi'$ and each $\pi\in\mathcal{S}(Q)$ satisfy certain unproven partial progress towards the generalized Ramanujan conjecture (GRC) \cite[Hypothesis 1.1 and Theorem 1.3]{BTZ}.\footnote{If $\theta_n$ in \eqref{eqn:LRS_finite} satisfies $\theta_{n}\leq\frac{1}{4}-\delta_n$ for some $\delta_n>0$, then each $\pi\in\mathfrak{F}_n$ satisfies \cite[Hypothesis 1.1]{BTZ}.}
\end{itemize}
Here, we prove the first completely unconditional zero density estimate of the form \eqref{eqn:goal}.

\begin{theorem}
	\label{thm:ZDE}
	Let $n,n'\geq 1$ and $\epsilon>0$.  Let $\mathcal{S}\subseteq\mathfrak{F}_n$ and $\mathcal{S}(Q)=\{\pi\in\mathcal{S}\colon C(\pi)\leq Q\}$.  If $0\leq\sigma\leq 1$, $\pi'\in\mathfrak{F}_{n'}$, and $Q,T\geq 1$, then
	\[
	\sum_{\pi\in\mathcal{S}(Q)}N_{\pi\times\pi'}(\sigma,T)\ll_{n,n',[F:\Q],\epsilon} \big(|\mathcal{S}(Q)|^{4} \big(C(\pi') Q  T^{[F:\Q]}\big)^{6.15\max\{n^2,n'n\}}\big)^{1-\sigma+\epsilon}.
	\]
\end{theorem}
\cref{thm:ZDE} is nontrivial when
\begin{equation}
\label{eqn:subfamily2}
\delta_{\mathcal{S}}=\liminf_{Q\to\infty}\frac{\log|\mathcal{S}(Q)|}{\log Q}>0.
\end{equation}
This is important because in applications, it is usually convenient to bound $Q$ by a power of $|\mathcal{S}(Q)|$ or vice versa.  When $\mathcal{S}=\mathfrak{F}_n$, we have the bounds
\begin{equation}
\label{eqn:poly_upper}
Q^{n+1}\ll_{n,F}|\mathfrak{F}_n(Q)|\ll_{\epsilon}D_F^{-n^2}Q^{2n+\epsilon}.
\end{equation}
The upper bound in \eqref{eqn:poly_upper} was proved by Brumley, Thorner, and Zaman \cite[Theorem A.1]{BTZ}.  The lower bound in \eqref{eqn:poly_upper} follows from work of Brumley and Mili{\'c}evi{\'c} \cite[Theorem 1.1]{BM}, who computed a constant $c_{n,F}>0$ such that if $\mathfrak{F}_n^*(Q)$ is the subset of $\pi\in\mathfrak{F}_n(Q)$ that are spherical at the archimedean places of $F$, then $|\mathfrak{F}_n^*(Q)|\sim c_{n,F}Q^{n+1}$.  (The lower bound in \eqref{eqn:poly_upper} reflects the conjectured order of growth; see \cite{BM}.)  Together, \cref{thm:ZDE} and \eqref{eqn:poly_upper} imply that
\begin{equation}
\label{eqn:ZDE_F}
\sum_{\pi\in\mathfrak{F}_n(Q)}N_{\pi\times\pi'}(\sigma,T)\ll_{n,n',F,\epsilon}(|\mathfrak{F}_n(Q)|C(\pi')^{n}T^{n[F:\Q]})^{7.1\max\{n,n'\}(1-\sigma)+\epsilon}.
\end{equation}
Furthermore, if $n'=n\geq 3$ and $\pi'\in\mathfrak{F}_{n'}$, then \cref{thm:ZDE} and \eqref{eqn:poly_upper} together imply that
\begin{equation}
\label{eqn:ZDE_Q}
\sum_{\pi\in\mathfrak{F}_n(Q)}N_{\pi\times\pi'}(\sigma,T)\ll_{\epsilon}(C(\pi') Q T^{[F:\Q]})^{9n^2(1-\sigma)+\epsilon}.
\end{equation}

\section{Applications}
\label{sec:Applications}

We will now describe some applications of \cref{thm:ZDE}.  In what follows, we write $f\ll_{\nu}g$, $f=O_{\nu}(g)$, and $g\gg_{\nu}f$ to denote that there exists a constant $c>0$ such that $|f|\leq c|g|$ in the stated range.  The implied constant $c$, which is effectively computable unless otherwise stated, will depend at most on $\nu$, $n$, $n'$, and $[F:\Q]$.  The expression $f\asymp_{\nu} g$ means that $f\ll_{\nu}g$ and $g\ll_{\nu}f$.  We use $\epsilon>0$ to denote an arbitrarily small quantity that depends at most on $n$, $n'$, and $[F:\Q]$.

\subsection{Bounds for Rankin--Selberg \texorpdfstring{$L$}{L}-functions}

It is a classical problem for Dirichlet $L$-functions to find strong bounds on the critical line $\re(s)=\frac{1}{2}$. The Phragm{\'e}n--Lindel{\"o}f convexity principle shows that if $q_{\chi}$ is the conductor of a primitive Dirichlet character $\chi$, then $L(\tfrac{1}{2},\chi) \ll q^{1/4}$; improving this bound by replacing $1/4$ with a smaller exponent is known as a \emph{subconvex} bound. The multiplicative version of the classical large sieve inequality combined with an approximate functional equation shows that for almost all $\chi$, we have the bound $L(\frac{1}{2},\chi)\ll_{\epsilon}q_{\chi}^{\epsilon}$ for all $\epsilon>0$, consistent with the generalized Lindel{\"o}f hypothesis (GLH).

For $\pi\in\mathfrak{F}_2$, GLH predicts that $L(\tfrac{1}{2},\pi)\ll_{\epsilon}C(\pi)^{\epsilon}$.  Michel and Venkatesh \cite{MV} proved that there exists a fixed positive $\delta>0$ such that $L(\frac{1}{2},\pi)\ll_F C(\pi)^{1/4-\delta}$, the culmination of several decades of research.  When $F=\Q$, a sharp mean value estimate for Hecke eigenvalues proved by Deshouillers and Iwaniec \cite{DI}, in conjunction with the approximate functional equation, implies the bound $L(\frac{1}{2},\pi)\ll_{\epsilon} (qT)^{\epsilon}$ for almost all $\pi\in\mathfrak{F}_2$ of arithmetic conductor $q$, trivial central character, and archimedean complexity (Laplace eigenvalue or weight squared) lying in the dyadic interval $[T,2T]$.  Note that in this case, $C(\pi)\asymp qT$.

For $\pi\in\mathfrak{F}_n$ with $n\geq 3$, the best uniform result towards the bound $L(\frac{1}{2},\pi)\ll_{\epsilon}C(\pi)^{\epsilon}$ predicted by GLH is that of Soundararajan and Thorner \cite[Corollary 2.7]{ST}, namely
\begin{equation}
\label{eqn:ST_weak}
L(\tfrac{1}{2},\pi)\ll C(\pi)^{\frac{1}{4}}(\log C(\pi))^{-1/(10^{17}n^3)}.
\end{equation}
We mention three results that improve upon \eqref{eqn:ST_weak} in an average sense, each having complementary strengths.  Jana \cite[Theorem 6]{jana2020applications} extended the GLH-on-average bound of Deshouillers and Iwaniec to the family of cuspidal automorphic representations of $\GL_n(\A_{\Q})$ of arithmetic conductor 1 and growing analytic conductor.  Blomer \cite[Corollary 5]{Blomer} proved the corresponding result for the family of cuspidal automorphic representations of $\GL_n(\A_{\Q})$ of a large given prime arithmetic conductor $q$, trivial central character, and whose archimedean components are principal series representations confined to a compact subset of the unitary dual.  Thorner and Zaman \cite[Theorem 1.3]{TZ_GLn} proved that there exists a constant $\Cl[abcon]{TZ_gln} = \Cr{TZ_gln}(n,[F:\Q])>0$ such that if $\epsilon>0$, then
\begin{equation}
\label{eqn:tzlargesievedensity}
|\{\pi\in\mathfrak{F}_n(Q)\colon |L(\tfrac{1}{2},\pi)|\geq \Cr{TZ_gln} C(\pi)^{\frac{1}{4}-\frac{\epsilon}{10^{16}n^3}}\}|\ll_{F}|\mathfrak{F}_n(Q)|^{\epsilon}.
\end{equation}
Unlike the preceding results, \eqref{eqn:tzlargesievedensity} is uniform in both the arithmetic conductor and spectral aspects and holds over number fields other than $\Q$, but the savings over \eqref{eqn:ST_weak} is not comparable to GLH on average.

Given $\pi\in\mathfrak{F}_n$ and $\pi'\in\mathfrak{F}_{n'}$, Soundararajan and Thorner \cite[Corollary 2.7]{ST} proved when $F=\Q$ that if $C(\pi\times\pi')$ is the analytic conductor of $L(s,\pi\times\pi')$, then
\begin{equation}
\label{eqn:soundthornerRS}
L(\tfrac{1}{2},\pi\times\pi')\ll |L(\tfrac{3}{2},\pi\times\pi')|^2 C(\pi\times\pi')^{\frac{1}{4}} (\log C(\pi\times\pi'))^{-1/(10^{17}(n'n)^3)}.
\end{equation}
As of now, the best general upper bound for $|L(\tfrac{3}{2},\pi\times\pi')|^2$ is larger than any fixed power of $\log C(\pi\times\pi')$ \cite[Theorem 2]{Li}.  The factor of $|L(\tfrac{3}{2},\pi\times\pi')|^2$ can be removed under certain partial progress toward GRC.  The bound $L(\frac{1}{2},\pi\times\pi')\ll_{\epsilon}C(\pi\times\pi')^{\epsilon}$ is predicted by GLH.

In order to improve \eqref{eqn:soundthornerRS} on average with uniformity in $\pi$ and $\pi'$, one might first try to mimic the approach that worked well for Dirichlet $L$-functions and $\GL_2$ $L$-functions using trace formulae, approximate functional equations, the spectral large sieve, Vorono\u{\i} summation, etc. While such methods have seen great success for $\GL_n\times\GL_{n'}$ with $n,n' \in \{1,2\}$, suitably uniform and flexible versions of these tools do not appear to be available yet in the general setting.  The special case where $|n-n'|\leq 1$ exhibits some nice structural properties, lending itself to approaches via period integrals that completely avoids the aforementioned tools.  To describe work in this direction, let $F=\Q$, $\mathcal{F}_n\subseteq\mathfrak{F}_n$ be the subset of cuspidal automorphic representations of $\mathrm{GL}_n(\A_{\Q})$ of arithmetic conductor 1, and $\mathcal{F}_n(Q)=\{\pi\in\mathcal{F}_n\colon C(\pi)\leq Q\}$.  It follows from work of Jana \cite[Corollary 2.2]{Jana} that if $\epsilon>0$, $\pi'\in\mathcal{F}_n$, and the spectral parameters of $\pi'$ have real part at least $-\frac{1}{n^2+1}$ (which is far stronger than the best known unconditional lower bound $-\frac{1}{2}+\frac{1}{n^2+1}$ due to Luo, Rudnick, and Sarnak \cite{LRS}), then
\begin{equation}
\label{eqn:Jana}
\sum_{\pi\in\mathcal{F}_n(Q)}|L(\tfrac{1}{2},\pi\times\pi')|^2\ll_{\pi',\epsilon}|\mathcal{F}_n(Q)|^{1+\epsilon}.
\end{equation}
Therefore, for fixed $\pi'\in\mathcal{F}_n$, the generalized Lindel{\"o}f hypothesis for $L(\frac{1}{2},\pi\times \pi')$ holds on average over the $\pi\in\mathcal{F}_n$.  Using Chebyshev's inequality, we conclude that for all $\delta>0$, there exists a constant $c_{\pi',\delta}>0$ such that
\begin{equation}
\label{eqn:Jana2}
|\{\pi\in\mathcal{F}_n(Q)\colon |L(\tfrac{1}{2},\pi\times \pi')|\geq c_{\pi',\delta} C(\pi\times \pi')^{2\delta}\}|\ll_{\pi',\delta}|\mathcal{F}_n(Q)|^{1-\delta}.	
\end{equation}
See also the work of Blomer \cite[Theorem 2]{Blomer_RS}, which proves a variant of \eqref{eqn:Jana} for families of ``spectrally close'' Hecke--Maa{\ss} newforms on $\mathrm{SL}_n(\Z)$.

Along the same lines as \eqref{eqn:tzlargesievedensity}, we use \eqref{eqn:ZDE_F} to prove the following result.

\begin{theorem}
\label{thm:subconvexity}
Let $n,n'\geq 1$ and $Q\geq 1$.  If $\epsilon>0$ and $\pi'\in\mathfrak{F}_{n'}$, then
\[
|\{\pi\in\mathfrak{F}_n(Q)\colon |L(\tfrac{1}{2},\pi\times\pi')|\geq C(\pi\times\pi')^{\frac{1}{4}-\frac{\epsilon}{10^{10}\max\{n,n'\}}}\}|\ll_{\epsilon}(C(\pi')^{n} |\mathfrak{F}_n(Q)|)^{\epsilon}.
\]
\end{theorem}
\begin{remark}
Given a subset $\mathcal{S}\subseteq\mathfrak{F}_n$, a similar result can be proved for $\pi\in\mathcal{S}(Q)$ using \cref{thm:ZDE}.  Such a result would depend effectively on $\delta_{\mathcal{S}}$ in \eqref{eqn:subfamily2}.
\end{remark}

In contrast with Jana's work in \eqref{eqn:Jana2}, \cref{thm:subconvexity} provides a smaller power-saving improvement over \eqref{eqn:soundthornerRS}, but the improvement is uniform in the arithmetic conductor and spectral aspects as well as in $\pi'$.  The exceptional set in \cref{thm:subconvexity} is a much smaller than in \eqref{eqn:Jana2}.  \cref{thm:subconvexity} removes the requirements that $n=n'$, that $q_{\pi}=q_{\pi'}=1$, and that the spectral parameters of $\pi'$ have real part at least $-\frac{1}{n^2+1}$.  Finally, \cref{thm:subconvexity} is proved over any number field, while \cite{Jana} is only proved over $\Q$.

\subsection{Effective multiplicity one}
\label{subsec:mult_one}

Let $\pi=\bigotimes_v \pi_v$ and $\pi'=\bigotimes_v \pi_v'$ be cuspidal automorphic representations in $\mathfrak{F}_n(Q)$. Under the assumption of GRH for $L(s,\pi\times\tilde{\pi})$ and $L(s,\pi\times\pi')$ and that $\pi_{\kp}$ and $\pi_{\kp}'$ are tempered for all prime ideals $\kp \mid \kq_{\pi}\kq_{\pi'}$, it is known that if $Y=(\log Q)^2$ and $\pi_{\kp}\cong\pi_{\kp}'$ for all $\kp\nmid \kq_{\pi}\kq_{\pi'}$ with $\N\kp\ll Y$, then $\pi=\pi'$ \cite[Proposition 5.22]{IK}. Brumley \cite{Brumley}, improving on work of Moreno \cite{Moreno}, proved that there exists a constant $B_n>0$ such that this result holds unconditionally with $Y=Q^{B_n}$.  This result makes effective the multiplicity one theorems of Jacquet and Shalika \cite[Theorem 4.8]{JS} and Piatetski-Shapiro \cite{MR546599}.  Any fixed $B_n>2n$ suffices \cite{MR2522710}.

When $n=2$, we have an average result that nearly achieves what GRH predicts.  Specifically, let $\mathfrak{F}_2^{\flat}$ be the subset of $\pi\in\mathfrak{F}_2$ with squarefree conductor and trivial central character, and let $\pi'\in\mathfrak{F}_2^{\flat}$. For all $\epsilon>0$, there exists an effectively computable constant $N_{\epsilon}>0$, depending at most on $\epsilon$ and $[F:\Q]$, such that
\begin{equation}
\label{eqn:DK}
|\{\pi\in\mathfrak{F}_2^{\flat}(Q)\colon \textup{$\pi_{\kp}\cong\pi_{\kp}'$ for all $\kp\nmid\kq_{\pi}$ with $\N\kp\leq (\log Q)^{N_{\epsilon}}$}\}|\ll_{\epsilon} Q^{\epsilon}.
\end{equation}
In particular, the implied constant does not depend on $\pi'$.  This was proved by Duke and Kowalski \cite[Theorem 3]{DK} when $F=\Q$ in a stronger form under the assumption of GRC.  See Brumley's Ph.D.\ thesis \cite[Corollary 5.2.2]{Brumley_Thesis} for a proof that does not use GRC.

If $\pi\in\mathfrak{F}_2$ and $\tilde{\pi}\in\mathfrak{F}_2$ is the contragredient, then $L(s,\pi\times\tilde{\pi})=\zeta_F(s)L(s,\pi,\Ad)$, where $\zeta_F(s)$ is the Dedekind zeta function of $F$ and $\mathrm{Ad}$ is the adjoint square lift from a representation of $\GL_2(\A_F)$ to a representation of $\GL_3(\A_F)$.  The arguments in \cite{Brumley_Thesis,DK} rely on two key results:
\begin{enumerate}
	\item Gelbart and Jacquet \cite{GJ} proved that if $\pi\in\mathfrak{F}_2$, then $L(s,\pi,\Ad)$ is the $L$-function of an automorphic representation of $\GL_3(\A_F)$, denoted $\Ad\,\pi$, complete with a criterion by which one can determine whether $\mathrm{Ad}\,\pi$ is cuspidal (and lies in $\mathfrak{F}_3$).
	\item At most $O_{\epsilon}(Q^{1/2+\epsilon})$ representations $\pi\in\mathfrak{F}_2(Q)$ have the same adjoint lift.  Also, Ramakrishnan (\cite[Appendix]{DK} and \cite{Ramakrishnan}) proved that $\mathrm{Ad}\colon \mathfrak{F}_2^{\flat}\to\mathfrak{F}_3$ is injective.
\end{enumerate}

In attempting to generalize the strategy of Duke and Kowalski to $\GL_n$ for $n\geq 3$, one encounters some deep open problems.  If $\pi\in\mathfrak{F}_n$, then for $\re(s)$ sufficiently large, $L(s,\pi\times\tilde{\pi})$ factors as $\zeta_F(s)L(s,\pi,\mathrm{Ad}\,)$, where $\mathrm{Ad}\,$ is the adjoint square lift from $\GL_{n}$ to $\GL_{n^2-1}$.  Apart from some special cases, the following obstacles arise:
\begin{enumerate}
	\item The adjoint lift is not yet known to be automorphic for $n\geq 3$, and if it were, there is no known criterion for cuspidality.
	\item Let $\mathcal{H}_n\subseteq\mathfrak{F}_n(Q)$ be the subset of $\pi\in\mathfrak{F}_n(Q)$ such that $\mathrm{Ad}\,\pi\in\mathfrak{F}_{n^2-1}$.  It is not known how many $\pi\in\mathcal{H}_n$ have the same adjoint lift.
\end{enumerate}

Despite these setbacks, we can use \eqref{eqn:ZDE_Q} (more specifically, \cref{cor:good_ZFR} below) to prove a $\GL_n$ analogue of \eqref{eqn:DK} when $\pi'$ depends mildly on $Q$.
\begin{theorem}
	\label{thm:multiplicity_one}
Let $n\geq 3$ and $Q\geq 1$.  There exists an absolute, effectively computable constant $\Cl[abcon]{multone_1}>0$ such that if $0<\epsilon<1$ and $\pi'\in\mathfrak{F}_n((\log Q)^{\Cr{multone_1}/(n^2 [F:\Q]^2)})$, then
\[
|\{\pi\in\mathfrak{F}_n(Q)\colon \pi_{\kp}\cong\pi_{\kp}'~\textup{for all}~\kp\nmid\kq_{\pi}\kq_{\pi'}~\textup{with}~\N\kp\leq (\log Q)^{41n^2/\epsilon}\}|\ll_{\epsilon}Q^{\epsilon}.
\]
\end{theorem}
\begin{remark}
Given a subset $\mathcal{S}\subseteq\mathfrak{F}_n$, a similar result can be proved for $\pi\in\mathcal{S}(Q)$ using \cref{thm:ZDE}.  Such a result would depend effectively on $\delta_{\mathcal{S}}$ in \eqref{eqn:subfamily2}.
\end{remark}

The proof is very flexible.  For example, if $\epsilon>0$ and $C(\pi')\ll_{\epsilon}Q^{\Cr{multone_1}\epsilon^2/(n^2[F:\Q]^2)}$, then the same proof with minor changes in choices of parameters produces the bound
\[
|\{\pi\in\mathfrak{F}_n(Q)\colon \pi_{\kp}\cong\pi_{\kp}'~\textup{for all}~\kp\nmid\kq_{\pi}\kq_{\pi'}~\textup{with}~\N\kp\leq Q^{\epsilon}\}|\ll_{\epsilon}Q^{\epsilon}.
\]
While the number of $\kp$ for which one needs to check that $\pi_{\kp}\cong\pi_{\kp}'$ is larger, the range of $C(\pi')$ is greatly extended, and the threshold $\N\kp\leq Q^{\epsilon}$ (reminiscent of Vinogradov's conjecture on the size of the least quadratic nonresidue) still greatly improves on the unconditional range $\N\kp\ll_{\epsilon}Q^{2n+\epsilon}$ from \cite{MR2522710}.

\subsection{Automorphic level of distribution}

Let $\Lambda(m)$ be the von Mangoldt function, equal to $\log p$ if $m$ is a power of a prime $p$ and zero otherwise.  The celebrated Bombieri--Vinogradov theorem states that if $\theta<\frac{1}{2}$ is fixed, then for all $A>0$, we have
\begin{equation}
\label{eqn:BV}
\sum_{q\leq x^{\theta}}~\max_{\gcd(a,q)=1}\max_{y\leq x}\Big|\sum_{\substack{m\leq y \\ m\equiv a\pmod{q}}}\Lambda(m)-\frac{y}{\varphi(q)}\Big|\ll_A\frac{x}{(\log x)^A}.
\end{equation}
This may be viewed as an average form of GRH for Dirichlet $L$-functions.  As part of his proof \cite{Bombieri}, Bombieri proved a strong form of the zero density estimate in \cref{thm:ZDE} for Dirichlet $L$-functions.  We call any $\theta$ for which \eqref{eqn:BV} holds a {\it level of distribution} for the primes.  Elliott and Halberstam conjectured that any fixed $\theta<1$ is a level of distribution for the primes.

Number theorists have proved several interesting extensions and variations of \eqref{eqn:BV}.  For example, Murty and Murty \cite{MM} proved that primes in the Chebotarev density theorem have a positive level of distribution.  To describe a different direction for automorphic representations over $\Q$, we let $n\geq 2$ and consider $\pi\in\mathfrak{F}_{n}$ with conductor $q_{\pi}$.  Let $\Lambda(m)$ be the von Mangoldt function, and define the numbers $a_{\pi}(m)$ by
\[
-\frac{L'}{L}(s,\pi)=\sum_{p}\sum_{k=1}^{\infty}\frac{\sum_{j=1}^{n}\alpha_{j,\pi}(p)^k\log p}{p^{ks}}=\sum_{n=1}^{\infty}\frac{a_{\pi}(m)\Lambda(m)}{m^s},\qquad \re(s)>1.
\]
Note that $a_{\pi}(p)=\lambda_{\pi}(p)$.  For fixed $\theta<\frac{1}{n^2-2}$, Wong \cite[Theorem 9]{Wong} proved that if $\pi$ satisfies GRC and $L(s,\pi\times(\tilde{\pi}\otimes\chi))$ has no Landau--Siegel zero for all Dirichlet characters $\chi$, then for any $A>0$, 
\begin{equation}
\label{eqn:BV_Wong}
\sum_{q\leq x^{\theta}}~\max_{\gcd(a,q)=1}\max_{y\leq x}\Big|\sum_{\substack{m\leq y \\ \gcd(m,q_{\pi})=1 \\ m\equiv a\pmod{q}}}|a_{\pi}(m)|^2\Lambda(m)-\frac{y}{\varphi(q)}\Big|\ll_{A,\pi}\frac{x}{(\log x)^A}.
\end{equation}
This conditionally endows $L(s,\pi\times\tilde{\pi})$ with a positive level of distribution $\theta$.  The hypotheses for \eqref{eqn:BV_Wong} hold for $\pi$ attached to non-CM holomorphic cuspidal newforms on congruence subgroups of $\SL_2(\Z)$.

Let $\pi\in\mathfrak{F}_n$.  Using \eqref{eqn:ZDE_Q}, we unconditionally endow $L(s,\pi\times\tilde{\pi})$ with a notion of positive level of distribution.  In particular, we avoid recourse to unproven progress toward GRC or the absence of Landau--Siegel zeros.

\begin{theorem}
	\label{thm:BV}
	Let $F=\Q$ and $\pi\in\mathfrak{F}_{n}$.  Fix $\theta<1/(9n^3)$.  If $A>0$, then
	\[
	\sum_{\substack{q\leq x^{\theta} \\ \gcd(q,q_{\pi})=1}}\max_{\gcd(a,q)=1}\max_{y\leq x}\Big|\sum_{\substack{m\leq y \\ m\equiv a\pmod{q}}}|a_{\pi}(m)|^2\Lambda(m)-\frac{y}{\varphi(q)}\Big|\ll_{A,\pi}\frac{x}{(\log x)^A}.
	\]
	The implied constants are ineffective.
\end{theorem}

\begin{remark}
One can prove an analogue of \cref{thm:BV} with $F\neq\Q$, replacing residue classes modulo $q$ with ray classes modulo $\kq$.  We restrict to $F=\Q$ for notational simplicity.
\end{remark}

\subsection*{Overview of the paper}  In \cref{sec:L-functions}, we recall basic properties of standard $L$-functions and Rankin--Selberg $L$-functions that we will use in our proofs.  In \cref{sec:large_sieve}, we prove a large sieve inequality for the Dirichlet coefficients of $L(s,\pi\times\pi')^{-1}$ and a corollary on mean values of Dirichlet polynomials, which we use in our proof of \cref{thm:ZDE} in \cref{sec:ZDE}.   We then prove \cref{thm:subconvexity} in \cref{sec:good_ZFR}, \cref{thm:multiplicity_one} in  \cref{sec:multiplicity_one}, and \cref{thm:BV} in \cref{sec:BV}.

\subsection*{Acknowledgements}

We thank Gergely Harcos, Rizwanur Khan, and the anonymous referee for helpful comments.

\section{Properties of \texorpdfstring{$L$}{L}-functions}
\label{sec:L-functions}

We recall some standard facts about $L$-functions arising from automorphic representations and their Rankin--Selberg convolutions.  See \cite{Brumley,GJ2,JPSS,MW,ST}.

\subsection{Standard \texorpdfstring{$L$}{L}-functions}

Given $\pi\in\mathfrak{F}_n$, let $\widetilde{\pi}\in\mathfrak{F}_n$ be the contragredient representation and $\kq_{\pi}$ be the conductor of $\pi$.  We express $\pi$ as a restricted tensor product $\bigotimes_v \pi_v$ of smooth admissible representations of $\GL_n(F_v)$, where $v$ varies over places of $F$.  When $v$ is a nonarchimedean place corresponding with a prime ideal $\kp$, then the local $L$-function $L(s,\pi_{\kp})$ is defined in terms of the Satake parameters $A_{\pi}(\kp)=\{\alpha_{1,\pi}(\kp),\ldots,\alpha_{n,\pi}(\kp)\}$ by
\begin{equation}
	\label{eqn:Euler_p_single}
	L(s,\pi_{\kp})=\prod_{j=1}^{n}(1-\alpha_{j,\pi}(\kp)\N\kp^{-s})^{-1}=\sum_{k=0}^{\infty}\frac{\lambda_{\pi}(\kp^k)}{\N\kp^{ks}}.
\end{equation}
We have $\alpha_{j,\pi}(\kp)\neq0$ for all $j$ whenever $\kp\nmid\kq_{\pi}$, and when $\kp \mid \kq_{\pi}$, it might be the case that there exist $j$ such that $\alpha_{j,\pi}(\kp)=0$.  The standard $L$-function $L(s,\pi)$ associated to $\pi$ is of the form
\[
L(s,\pi)=\prod_{\kp} L(s,\pi_{\kp})=\sum_{\kn}\frac{\lambda_{\pi}(\kn)}{\N\kn^s}.
\]
The Euler product and Dirichlet series converge absolutely when $\re(s)>1$.

At each archimedean place $ v$ of $F$, there are $n$ Langlands parameters $\mu_{j,\pi}(v)\in\mathbb{C}$ such that
\[
L(s,\pi_{\infty}) = \prod_{v|\infty}\prod_{j=1}^{n}\Gamma_{ v}(s+\mu_{j,\pi}(v)),\qquad \Gamma_{v}(s)\coloneqq \begin{cases}
	\pi^{-s/2}\Gamma(s/2)&\mbox{if $F_{ v}=\R$,}\\
	2(2\pi)^{-s}\Gamma(s)&\mbox{if $F_{ v}=\mathbb{C}$.}
\end{cases}
\]
By combining the work in \cite{BB2,BB,LRS,MS}, we know that there exists
\begin{equation}
\label{eqn:ramanujan_progress}
0\leq\theta_n\leq\begin{cases}
0&\mbox{if $n=1$,}\\
7/64&\mbox{if $n=2$,}\\
5/14&\mbox{if $n=3$,}\\
9/22&\mbox{if $n=4$,}\\
1/2-1/(n^2+1)&\mbox{if $n\geq 5$}
\end{cases}
\end{equation}
such that
\begin{equation}
\label{eqn:LRS_finite}
	|\alpha_{j,\pi}(\kp)|\leq  \N\kp^{\theta_n}\qquad\textup{ and }\qquad\re(\mu_{j,\pi}(v))\geq -\theta_n.
\end{equation}
GRC asserts that in \eqref{eqn:ramanujan_progress}, one may take $\theta_n=0$.  We have $\kq_{\pi}=\kq_{\widetilde{\pi}}$, and for each $\kp$ and each $ v$, we have the equalities of sets $\{\alpha_{j,\widetilde{\pi}}(\kp)\}=\{\overline{ \alpha_{j,\pi}(\kp)}\}$ and $\{\mu_{j,\widetilde{\pi}}(v)\}=\{\overline{\mu_{j,\pi}(v)}\}$.

Let $r_{\pi}$ be the order of the pole of $L(s,\pi)$ at $s=1$.  The completed $L$-function
\[
\Lambda(s,\pi) = (s(s-1))^{r_{\pi}}(D_F^n \N\kq_{\pi})^{s/2}L(s,\pi)L(s,\pi_{\infty})
\]
is entire of order 1, and there exists a complex number $W(\pi)$ of modulus 1 such that for all $s\in\mathbb{C}$, we have the functional equation $\Lambda(s,\pi)=W(\pi)\Lambda(1-s,\widetilde{\pi})$.  Let $d(v)=1$ if $F_{ v}=\R$ and $d(v)=2$ if $F_{ v}=\mathbb{C}$.  The analytic conductor of $\pi$ \cite{IS} is given by
\begin{equation}
\label{eqn:analytic_conductor_def}
C(\pi,t)\coloneqq D_F^n \N\kq_{\pi}\prod_{v|\infty}\prod_{j=1}^n(3+|it+\mu_{j,\pi}(v)|^{d(v)}),\qquad C(\pi)\coloneqq C(\pi,0).
\end{equation}
Since $\Lambda(s,\pi)$ is entire of order 1, there exist complex numbers $a_{\pi}$ and $b_{\pi}$ such that
\[
\Lambda(s,\pi)=e^{a_{\pi}+b_{\pi}s}\prod_{\Lambda(\rho,\pi)=0}\Big(1-\frac{s}{\rho}\Big)e^{s/\rho}.
\]
The zeros $\rho$ in the above Hadamard product are the nontrivial zeros of $L(s,\pi)$, and the zeros of $L(s,\pi)$ that arise as poles of $s^{r_{\pi}}L(s,\pi_{\infty})$ are the trivial zeros.

\subsection{Rankin--Selberg \texorpdfstring{$L$}{L}-functions}
\label{subsec:RS}

Let $\pi\in\mathfrak{F}_n$ and $\pi'\in\mathfrak{F}_{n'}$.  At each prime ideal $\kp$, Jacquet, Piatetski-Shapiro, and Shalika \cite{JPSS} associate to $\pi_{\kp}$ and $\pi_{\kp}'$ a local Rankin--Selberg $L$-function
\begin{equation}
\label{eqn:RS_Dirichlet_series}
L(s,\pi_{\kp}\times\pi_{\kp}')=\prod_{j=1}^{n}\prod_{j'=1}^{n'}(1-\alpha_{j,j',\pi\times\pi'}(\kp) \N\kp^{-s})^{-1}=\sum_{k=0}^{\infty}\frac{\lambda_{\pi\times\pi'}(\kp^k)}{\N\kp^{ks}}
\end{equation}
and a local conductor $\kq_{\pi_{\kp}\times\pi_{\kp}'}$.  If $\kp\nmid \kq_{\pi}\kq_{\pi'}$, then we have the equality of sets
\begin{equation}
\label{eqn:separate_dirichlet_coeffs}
\{\alpha_{j,j',\pi\times\pi'}(\kp)\}=\{\alpha_{j,\pi}(\kp)\alpha_{j',\pi'}(\kp)\}.
\end{equation}
The Rankin--Selberg $L$-function $L(s,\pi\times\pi')$ associated to $\pi$ and $\pi'$ and its arithmetic conductor are
\[
L(s,\pi\times\pi')=\prod_{\kp}L(s,\pi_{\kp}\times\pi_{\kp}')=\sum_{\kn}\frac{\lambda_{\pi\times\pi'}(\kn)}{\N\kn^s},\qquad \kq_{\pi\times\pi'}=\prod_{\kp}\kq_{\pi_{\kp}\times\pi_{\kp}'}.
\]
At an archimedean place $v$ of $F$, Jacquet, Piatetski-Shapiro, and Shalika associate $n'n$ complex Langlands parameters $\mu_{j,j',\pi\times\pi'}(v)$ to $\pi_v$ and $\pi_v'$, from which one defines
\[
L(s,\pi_{\infty}\times\pi_{\infty}') = \prod_{v|\infty}\prod_{j=1}^{n}\prod_{j'=1}^{n'}\Gamma_{ v}(s+\mu_{j,j',\pi\times\pi'}(v)).
\]
Using the explicit descriptions of $\alpha_{j,j',\pi\times\pi'}(\kp)$ and $\mu_{j,j',\pi\times\pi'}(v)$ in \cite{Humphries,ST}, one sees that
\begin{equation}
\label{eqn:LRS_2}
|\alpha_{j,j',\pi\times\pi'}(\kp)|\leq\N\kp^{\theta_n + \theta_{n'}},\qquad \re(\mu_{j,j',\pi\times\pi'}(v))\geq -\theta_n - \theta_{n'}.
\end{equation}

Let $r_{\pi\times\pi'} = -\mathrm{ord}_{s=1}L(s,\pi\times\pi')$.  By our normalization for the central characters of $\pi$ and $\pi'$, we have that $r_{\pi\times\pi'}=0$ if and only if $\pi\neq \widetilde{\pi}'$, and $r_{\pi\times\widetilde{\pi}}=1$ otherwise.  The function
\begin{equation}
\label{eqn:Lambdaspixpi'}
\Lambda(s,\pi\times\pi')=(s(s-1))^{r_{\pi\times\pi'}}(D_F^{n'n}\N\kq_{\pi\times\pi'})^{s/2}L(s,\pi\times\pi')L(s,\pi_{\infty}\times\pi_{\infty}')
\end{equation}
is entire of order 1, and there exists a complex number $W(\pi\times\pi')$ of modulus 1 such that $\Lambda(s,\pi\times\pi')$ satisfies the functional equation $\Lambda(s,\pi\times\pi')=W(\pi\times\pi')\Lambda(1-s,\widetilde{\pi}\times\widetilde{\pi}')$.  As with $L(s,\pi)$, the analytic conductor of $L(s,\pi\times\pi')$ is given by
\begin{equation}
\label{eqn:analytic_conductor_def_2}
C(\pi\times\pi',t)\coloneqq D_F^{n'n}\N\kq_{\pi\times\pi'}\prod_{v|\infty}\prod_{j=1}^n \prod_{j'=1}^{n'}(3+|it+\mu_{j,j',\pi\times\pi'}(v)|^{d(v)}),\qquad C(\pi\times\pi')\coloneqq C(0,\pi\times\pi').
\end{equation}
The combined work of Bushnell and Henniart \cite{BH} and Brumley \cite[Appendix]{Humphries} yields
\begin{equation}
\label{eqn:BH}
C(\pi\times\pi',t)\ll C(\pi\times\pi')(3+|t|)^{[F:\Q] n'n},\qquad C(\pi\times\pi')\ll C(\pi)^{n'}C(\pi')^{n}.
\end{equation}
Since $\Lambda(s,\pi\times\pi')$ is entire of order 1, there exist complex numbers $a_{\pi\times\pi'}$ and $b_{\pi\times\pi'}$ such that the Hadamard factorization
\begin{equation}
\label{eqn:hadamard}
\Lambda(s,\pi\times\pi')=e^{a_{\pi\times\pi'}+b_{\pi\times\pi'}s}\prod_{\Lambda(\rho,\pi\times\pi')=0}\Big(1-\frac{s}{\rho}\Big)e^{s/\rho}
\end{equation}
holds.  The zeros $\rho$ in \eqref{eqn:hadamard} are the nontrivial zeros of $L(s,\pi\times\pi')$, and the zeros of $L(s,\pi\times\pi')$ that arise as poles of $s^{r_{\pi\times\pi'}}L(s,\pi_{\infty}\times\pi_{\infty}')$ are the trivial zeros.

It follows from work of Li \cite[Theorem 2]{Li} (with minor adjustments when $F\neq \Q$) that there exists an absolute and effectively computable constant $\Cl[abcon]{Li}>0$, which we assume to be sufficiently large for future convenience, such that
\begin{equation}
\label{eqn:Li_sharper}
\lim_{\sigma_0\to\sigma}(\sigma_0-1)^{r_{\pi\times\pi'}}L(\sigma_0,\pi\times\pi')\ll \exp\Big(\Cr{Li}n'n[F:\Q]\frac{\log C(\pi\times\pi')}{\log\log C(\pi\times\pi')}\Big),\quad \sigma\in[1,3].
\end{equation}
We can change $\pi'$ to $\pi'\otimes \left|\det\right|^{it}$; at the archimedean places, this has the effect of adding $it$ to each $\mu_{j,j',\pi\times\pi'}(v)$.  We then apply functional equation, the Phragm{\'e}n--Lindel{\"o}f convexity principle, and \eqref{eqn:BH} to obtain for all $\sigma\geq 0$
\begin{equation}
\label{eqn:preconvex}
\begin{aligned}
&\lim_{\sigma_0\to\sigma}\Big(\frac{\sigma_0+it-1}{\sigma_0+it+1}\Big)^{r_{\pi\times\pi'}}L(\sigma_0+it,\pi\times\pi')\\
&\ll_{\epsilon} C(\pi\times\pi',t)^{\frac{\max\{1-\sigma,0\}}{2}+\frac{\epsilon}{n'n[F:\Q]}}\\
&\ll_{\epsilon} (C(\pi)^{n'} C(\pi')^n(3+|t|)^{n'n[F:\Q]})^{\frac{\max\{1-\sigma,0\}}{2}+\frac{\epsilon}{n'n[F:\Q]}}.
\end{aligned}
\end{equation}

\begin{lemma}
\label{lem:mertens}
	If $\pi\in\mathfrak{F}_n$, $X\geq 3$, and $\epsilon>0$, then $\sum_{\N\kn\leq X}\lambda_{\pi\times\widetilde{\pi}}(\kn)/\N\kn\ll_{\epsilon} C(\pi)^{\epsilon}\log X$.
\end{lemma}
\begin{proof}
	Since $\lambda_{\pi\times\tilde{\pi}}(\kn)\geq 0$ for all $\kn$ by \cite[Lemma a]{Hoffstein}, we observe by \eqref{eqn:Li_sharper} that
	\[
	\sum_{\N\kn\leq X}\frac{\lambda_{\pi\times\widetilde{\pi}}(\kn)}{\N\kn}\leq e \sum_{\kn}\frac{\lambda_{\pi\times\tilde{\pi}}(\kn)}{\N\kn^{1+\frac{1}{\log X}}}\ll (\log X)\mathop{\mathrm{Res}}_{s=1}L(s,\pi\times\tilde{\pi}).
	\]
	The desired bounded now follows from \eqref{eqn:preconvex} with $\sigma=1$, $t=0$, and $\pi'=\tilde{\pi}$.
\end{proof}

\begin{lemma}
\label{lem:GHL}
	Let $J\geq 1$ be an integer.  For all $j\in\{1,\ldots,J\}$, let $t_j\in\R$; $n_j,n_j'$ be positive integers; and $\pi_j\in\mathfrak{F}_{n_j}$ and $\pi_j'\in\mathfrak{F}_{n_j'}$.  Let
	\[
	D(s) = \prod_{j=1}^J L(s+it_j,\pi_j\times\tilde{\pi}_j'),\qquad \Delta(s) = \prod_{j=1}^J \Lambda(s+it_j,\pi_j\times\tilde{\pi}_j'),\qquad Q=\prod_{j=1}^J C(\pi_j)^{n_j'}C(\pi_j')^{n_j}.
	\]
	Let $R=-\mathrm{ord}_{s=1}D(s)$.  If $1<\sigma<2$ and the $\kn$-th Dirichlet coefficient of $-\frac{D'}{D}(s)$ is nonnegative when $\gcd(\kn,\prod_{j=1}^J \kq_{\pi_j}\kq_{\pi_j'})=\cO_F$, then
	\[
	\sum_{\Delta(\rho)=0}\re\Big(\frac{1}{\sigma-\rho}\Big) <\frac{R}{\sigma-1} + \sum_{\substack{1\leq j\leq J \\ \pi_j=\pi_j',~t_j\neq 0}}\frac{\sigma-1}{(\sigma-1)^2+t_j^2}+O(\log Q).
	\]
\end{lemma}
\begin{proof}
	Let $1<\sigma<2$ and
	\[
	D_{\infty}(s) = \prod_{j=1}^{J}L(s,(\pi_j)_{\infty}\times(\tilde{\pi}_j')_{\infty}),\qquad \mathfrak{q}_{D}=\prod_{j=1}^{J}\kq_{\pi_j\times\widetilde{\pi}_j'}.
	\]
	By comparing the logarithmic derivative of $\Delta(s)$ with the logarithmic derivative of its Hadamard factorization
	\[
	\Delta(s)=e^{a_{D}+b_{D}s}\prod_{\Delta(\rho)=0}\Big(1-\frac{s}{\rho}\Big)e^{s/\rho},
	\]
	we find that
	\[
	\sum_{\Delta(\rho)=0}\Big(\frac{1}{\sigma-\rho}+\frac{1}{\rho}\Big)+b_{D}=\frac{D'}{D}(\sigma)+\frac{\log\N\kq_{D}}{2}+\frac{R}{\sigma-1}+\sum_{\substack{1\leq j\leq J \\ \pi_j = \pi_j'}}\Big(\frac{1}{\sigma+it_j-1}+\frac{1}{\sigma+it_j}\Big)+\frac{D_{\infty}'}{D_{\infty}}(\sigma).
	\]
Since $\re(b_{D})=-\sum_{\Delta(\rho)=0}\re(\rho^{-1})$ \cite[Proposition 5.7(3)]{IK}, we take real parts and obtain
	\[
	\sum_{\Delta(\rho)=0}\re\Big(\frac{1}{\sigma-\rho}\Big)=\re\Big(\frac{D'}{D}(s)+\frac{\log\N\kq_{D}}{2}+\frac{R}{\sigma-1}+\sum_{\substack{1\leq j\leq J \\ \pi_j=\pi_j' \\ t_j\neq 0}}\frac{\sigma-1}{(\sigma-1)^2+t_j^2}+\frac{D_{\infty}'}{D_{\infty}}(s)\Big)+O(1).
	\]
	By \eqref{eqn:LRS_finite}, \eqref{eqn:LRS_2}, and \eqref{eqn:BH}, the bound $\re(\frac{D'}{D}(\sigma))\ll\log Q$ (resp.\ $\frac{D_{\infty}'}{D_{\infty}}(\sigma)\ll\log Q$) follows from our hypothesis on the Dirichlet coefficients of $-\frac{D'}{D}(s)$ (resp.\ Stirling's formula).
\end{proof}

\subsection{Rankin--Selberg combinatorics}
\label{sec:prelim_large_sieve}

A partition $\mu=(\mu_i)_{i=1}^{\infty}$ is a sequence of nonincreasing nonnegative integers $\mu_1\geq\mu_2\geq\cdots$ with only finitely many nonzero entries.  For a partition $\mu$, let $\ell(\mu)$ be the number of nonzero $\mu_i$, and let $|\mu|=\sum_{i=1}^{\infty} \mu_i$.  For a set $\{\alpha_{1},\ldots,\alpha_n \}$ of real numbers and a partition $\mu$ with $\ell(\mu)\leq n$, let $s_{\mu}(\{\alpha_1,\ldots,\alpha_n\})$ be the Schur polynomial $\det[(\alpha_{i}^{\lambda(j)+n-j})_{ij}] / \det[(\alpha_{i}^{n-j})_{ij}]$ associated to $\mu$.  If $|\mu|=0$, then $s_{\mu}(\{\alpha_1,\ldots,\alpha_n\})$ is identically one.  By convention, if $\ell(\mu)>n$, then $s_{\mu}(\{\alpha_1,\ldots,\alpha_n\})$ is identically zero.

Let $\pi\in\mathfrak{F}_n$ and $\pi'\in\mathfrak{F}_{n'}$.  By  \eqref{eqn:RS_Dirichlet_series}, \eqref{eqn:separate_dirichlet_coeffs}, and Cauchy's identity \cite[Theorem 38.1]{Bump_lie}, we have
\[
\sum_{k=0}^{\infty}\frac{\lambda_{\pi\times\pi'}(\kp^k)}{\N\kp^{ks}}=L(s,\pi_{\kp}\times\pi_{\kp}')=\sum_{\mu}\frac{s_{\mu}(A_{\pi}(\kp))s_{\mu}(A_{\pi'}(\kp))}{\N\kp^{s|\mu|}},\quad \kp\nmid\kq_{\pi}\kq_{\pi'},
\]
where the sum ranges over all partitions.  This yields
\[
\lambda_{\pi\times\pi'}(\kp^k) = \sum_{ |\mu|=k}s_{\mu}(A_{\pi}(\kp))s_{\mu}(A_{\pi'}(\kp)),\qquad \kp\nmid\kq_{\pi}\kq_{\pi'}.
\]
For an integral ideal $\kn$ with factorization $\kn=\prod_{\kp}\kp^{\mathrm{ord}_{\kp}(\kn)}$ (with $\mathrm{ord}_{\kp}(\kn)=0$ for all but finitely many $\kp$), the multiplicativity of $\lambda_{\pi\times\pi'}(\kn)$ tells us that if $\gcd(\kn,\kq_{\pi}\kq_{\pi'})=\cO_F$, then
\begin{align}
\label{n-prime2S}
\lambda_{\pi\times\pi'}(\kn)=\prod_{\kp}\lambda_{\pi\times\pi'}(\kp^{\mathrm{ord}_{\kp}(\kn)})=\sum_{(\mu_{\kp})_{\kp}\in\underline{\mu}[\kn]}\prod_{\kp}s_{\mu_{\kp}}(A_{\pi}(\kp))s_{\mu_{\kp}}(A_{\pi'}(\kp)),
\end{align}
where $(\mu_{\kp})_{\kp}$ denotes a sequence of partitions indexed by prime ideals and
\begin{equation}
\label{eqn:underline_mu_def}
\underline{\mu}[\kn]\coloneqq \{(\mu_{\kp})_{\kp}\colon |\mu_{\kp}|=\mathrm{ord}_{\kp}(\kn)\textup{ for all $\kp$}\}.
\end{equation}

Define the numbers $\mu_{\pi\times\pi'}(\kn)$ on unramified prime powers by
\[
\sum_{k=0}^{\infty}\frac{\mu_{\pi\times\pi'}(\kp^k)}{\N\kp^{ks}}=L(s,\pi_{\kp}\times\pi_{\kp}')^{-1}=\prod_{j=1}^{n} \prod_{j'=1}^{n'}(1-\alpha_{j,\pi}(\kp)\alpha_{j',\pi'}(\kp)\N\kp^{-s}),\qquad \kp\nmid\kq_{\pi}\kq_{\pi'}.
\]
By multiplicativity, this defines $\mu_{\pi\times\pi'}(\kn)$ when $\gcd(\kn,\kq_{\pi}\kq_{\pi'})=\mathcal{O}_F $.  For a partition $\mu=(\mu_i)_{i=1}^{\infty}$, let $\mu^*=(\mu_i^*)_{i=1}^{\infty}$ be the dual partition defined by $\mu_i^*=|\{j\colon \mu_j\geq i\}|$.  It follows from the dual Cauchy identity \cite[Chapter 38]{Bump_lie} and \eqref{eqn:separate_dirichlet_coeffs} that
\[
\sum_{k=0}^{\infty}\frac{\mu_{\pi\times\pi'}(\kp^k)}{\N\kp^{ks}}=\sum_{\mu}\frac{s_{\mu}(A_{\pi}(\kp))s_{\mu^*}(-A_{\pi'}(\kp))}{\N\kp^{|\mu|s}},\qquad \kp\nmid\kq_{\pi}\kq_{\pi'},
\]
where $-A_{\pi'}(\kp)=\{-\alpha_{1,\pi'}(\kp),\ldots,-\alpha_{n,\pi'}(\kp)\}$.  Hence we have
\begin{equation}
\label{eqn:moebius}
\mu_{\pi\times\pi'}(\kn)=\sum_{(\mu_{\kp})_{\kp}\in\underline{\mu}[\kn]}\prod_{\kp}s_{\mu_{\kp}}(A_{\pi}(\kp))s_{\mu^*_{\kp}}(-A_{\pi}(\kp)),\qquad \gcd(\kn,\kq_{\pi}\kq_{\pi'})=\mathcal{O}_F .
\end{equation}

\begin{lemma}
	\label{lem:dual_partition}
	If $\gcd(\kn,\kq_{\pi}\kq_{\pi'})=\mathcal{O}_F $, then we have $|\mu_{\pi\times\pi'}(\kn)|\leq\frac{1}{2}(\lambda_{\pi\times\tilde{\pi}}(\kn)+\lambda_{\pi'\times\tilde{\pi}'}(\kn))$.  
\end{lemma}
\begin{proof}
	We apply the inequality of arithmetic and geometric means to \eqref{eqn:moebius}:
	\[
	|\mu_{\pi\times\pi'}(\kn)|\leq \frac{1}{2}\Big(\sum_{(\mu_{\kp})_{\kp}\in\underline{\mu}[\kn]}\Big|\prod_{\kp}s_{\mu_{\kp}}(A_{\pi}(\kp))\Big|^2+\sum_{(\mu_{\kp})_{\kp}\in\underline{\mu}[\kn]}\Big|\prod_{\kp}s_{\mu^*_{\kp}}(-A_{\pi'}(\kp))\Big|^2\Big).
	\]
	The first sum equals $\lambda_{\pi\times\tilde{\pi}}(\kn)$ by \eqref{n-prime2S}.  For the second sum, note that since $|\mu|=|\mu^*|$, we have $(\mu_{\kp})_{\kp}\in\underline{\mu}[\kn]$ if and only if $(\mu_{\kp}^*)_{\kp}\in\underline{\mu}[\kn]$.  Hence by rearranging, we see that
\begin{equation}
\label{eqn:dual_a}
\sum_{(\mu_{\kp})_{\kp}\in\underline{\mu}[\kn]}\Big|\prod_{\kp}s_{\mu^*_{\kp}}(-A_{\pi'}(\kp))\Big|^2=\sum_{(\mu_{\kp})_{\kp}\in\underline{\mu}[\kn]}\Big|\prod_{\kp}s_{\mu_{\kp}}(-A_{\pi'}(\kp))\Big|^2=\lambda_{\pi\times\tilde{\pi}'}(\kn),
\end{equation}
where the last equality holds because $\alpha_{j,\pi'}(\kp)\overline{\alpha_{j',\pi'}(\kp)}=(-\alpha_{j,\pi'}(\kp))\overline{(-\alpha_{j',\pi'}(\kp))}$.
\end{proof}

\section{A new mean value estimate}
\label{sec:large_sieve}

Our proof of \cref{thm:ZDE} uses the following new mean value estimate for the Dirichlet coefficients of $L(s,\pi\times\pi')$ and $L(s,\pi\times\pi')^{-1}$.  Let $\mathcal{S}\subseteq\mathfrak{F}_n$, and let $\mathcal{S}(Q)=\{\pi\in\mathcal{S}\colon C(\pi)\leq Q\}$.
\begin{theorem}
\label{thm:main_theorem_1}
	Let $b$ be a complex-valued function supported on the integral ideals of $\cO_F$.  Let $n,n'\geq 1$, and let $\pi'\in\mathfrak{F}_{n'}$.  Let $Q,T\geq 1$, $\epsilon>0$, and $x\geq 1$.  Both
	\[
	\sum_{\pi\in\mathcal{S}(Q)}\Big|\sum_{\substack{\N\kn\in(x,xe^{1/T}] \\ \gcd(\kn,\kq_{\pi}\kq_{\pi'})=\mathcal{O}_F }}\mu_{\pi\times\pi'}(\kn)b(\kn)\Big|^2\qquad\text{and}\qquad\sum_{\pi\in\mathcal{S}(Q)}\Big|\sum_{\substack{\N\kn\in(x,xe^{1/T}] \\ \gcd(\kn,\kq_{\pi}\kq_{\pi'})=\mathcal{O}_F }}\lambda_{\pi\times\pi'}(\kn)b(\kn)\Big|^2
	\]
	are
	\[
	\ll_{\epsilon} Q^{\epsilon}\Big(\frac{x}{T}+Q^{4n^2\theta_n+n}T^{\frac{1}{2}n^2 [F:\Q]+\epsilon}|\mathcal{S}(Q)|\Big)\sum_{\substack{\N\kn\in(x,xe^{1/T}] \\ \gcd(\kn,\kq_{\pi'})=\cO_F}}\lambda_{\pi'\times\tilde{\pi}'}(\kn)|b(\kn)|^2,
	\]
	where $\theta_n\in[0,\frac{1}{2}-\frac{1}{n^2+1}]$ is the best known exponent towards GRC for $\pi\in\mathcal{S}$.
\end{theorem}

\begin{remark}
If $\pi'=\mathbbm{1}$ and $\mathcal{S}=\mathfrak{F}_n$, then \cref{thm:main_theorem_1} recovers \cite[Theorem 1.1]{TZ_GLn}\footnote{The factor $Q^{4n^2\theta_n+n}$ fixes a minor error in the proof of  \cite[Theorem 1.1]{TZ_GLn}, which led to a factor of $Q^{n^2+n}$ instead.}:
	\[
	\sum_{\pi\in\mathfrak{F}_n(Q)}\Big|\sum_{\substack{\N\kn\in(x,xe^{1/T}] \\ \gcd(\kn,\kq_{\pi})=\mathcal{O}_F }}\lambda_{\pi}(\kn)b(\kn)\Big|^2\ll_{\epsilon} Q^{\epsilon}\Big(\frac{x}{T}+Q^{4n^2\theta_n+n}T^{\frac{1}{2}n^2 [F:\Q]+\epsilon}|\mathfrak{F}_n(Q)|\Big)\sum_{\substack{\N\kn\in(x,xe^{1/T}] \\ \gcd(\kn,\kq_{\pi'})=\cO_F}}|b(\kn)|^2.
	\]
One could try to prove \cref{thm:main_theorem_1} starting with this, replacing $b(\kn)$ with $\lambda_{\pi'}(\kn)\mu_F(\kn)b(\kn)$ (where $\mu_F(\kn)$ is the $\kn$-th Dirichlet coefficient of $\zeta_F(s)^{-1}$), and try to recover a version of \cref{thm:main_theorem_1} with $\mu_{\pi\times\pi'}(\kn)$ replaced by $\lambda_{\pi}(\kn)\lambda_{\pi'}(\kn)\mu_F(\kn)$.  Note that $\mu_{\pi\times\pi'}(\kn)=\lambda_{\pi}(\kn)\lambda_{\pi'}(\kn)\mu_F(\kn)$ when $\kn$ is squarefree and coprime to $\kq_{\pi}\kq_{\pi'}$.  Otherwise, equality is not guaranteed.  If one wants to approximate $L(s,\pi\times\pi')$ with $\sum_{\textup{$\kn$ squarefree}}\lambda_{\pi}(\kn)\lambda_{\pi'}(\kn)\N\kn^{-s}$ and extend into the critical strip when $\pi,\pi'\in\mathfrak{F}_n$ and $n\geq 5$, then one must have progress towards GRC well beyond what is known unconditionally \cite{Brumley_2,DK}.  Such progress would then be a hypothesis for \cref{thm:ZDE}.
\end{remark}

\cref{thm:main_theorem_1} provides the first nontrivial unconditional mean value estimates of large sieve type for the Dirichlet coefficients $\lambda_{\pi\times\pi'}(\kn)$ or $\mu_{\pi\times\pi'}(\kn)$ for arbitrary $n$ and $n'$.  \cref{thm:main_theorem_1} follows from a more general result, \cref{prop:pre_large_sieve} below, for sequences of products of Schur polynomials evaluated on the set $A_{\pi}(\kp)$ of Satake parameters of $\pi$ at $\kp$.  We begin with a mean value estimate for the Satake parameters of $\pi$ as $\pi\in\mathcal{S}$ varies.  Let
\begin{equation}
\label{eqn:coeffics}
a: \bigcup_{x<\N\kn\leq xe^{1/T}}\underline{\mu}[\kn]\to\mathbb{C},\qquad \alpha:\mathcal{S}(Q)\to\mathbb{C}
\end{equation}
be functions that are not identically zero.  Their $\ell^2$ norms $\|a\|_2$ and $\|\alpha\|_2$ are defined by 
\[
\|a\|_2 = \Big(\sum_{x<\N\kn\leq xe^{1/T}}~\sum_{(\mu_{\kp})_{\kp}\in\underline{\mu}[\kn]}|a((\mu_{\kp})_{\kp})|^2\Big)^{1/2},\qquad \|\alpha\|_2=\Big(\sum_{\pi\in\mathcal{S}(Q)}|\alpha(\pi)|^2\Big)^{1/2}.
\]
For convenience, we define $\mathbf{1}_{(\kn,\kq)}$ to equal one when $\gcd(\kn,\kq)=\mathcal{O}_F$ and zero otherwise.
\begin{proposition}
	\label{prop:pre_large_sieve}
Let $x\geq 1$ and $Q,T\geq 1$.  Define
\[
C(Q,T,x)\coloneqq \sup_{\|a\|_2\neq 0}\frac{1}{\|a\|_2^2}\sum_{\pi\in\mathcal{S}(Q)}\Big|\sum_{\substack{x<\N\kn\leq xe^{1/T} \\ \gcd(\kn,\kq_{\pi})=\mathcal{O}_F }}\sum_{(\mu_{\kp})_{\kp}\in\underline{\mu}[\kn]}\Big[\prod_{\kp}s_{\mu_{\kp}}(A_{\pi}(\kp))\Big]a((\mu_{\kp})_{\kp})\Big|^2.
\]
We have the bound $C(Q,T,x)\ll_{\epsilon} Q^{\epsilon}(T^{-1}x+Q^{4n^2\theta_n+n} T^{\frac{n^2[F:\Q]}{2}+\epsilon}|\mathcal{S}(Q)|)$.
\end{proposition}
\begin{proof}
We observe that
\begin{equation}
\label{eqn:step1}
\begin{aligned}
C(Q,T,x)=\sup_{\|a\|_2=1}\sum_{\pi\in\mathcal{S}(Q)}\Big|\sum_{\substack{x<\N\kn\leq xe^{1/T}}}~\sum_{(\mu_{\kp})_{\kp}\in\underline{\mu}[\kn]}\Big[\prod_{\kp}s_{\mu_{\kp}}(A_{\pi}(\kp))\Big]\mathbf{1}_{(\kn,\kq_{\pi})}a((\mu_{\kp})_{\kp})\Big|^2.
\end{aligned}
\end{equation}
By the duality principle for bilinear forms, \eqref{eqn:step1} is bounded by the supremum over the functions $\alpha\colon \mathcal{S}(Q)\to\mathbb{C}$ such that $\|\alpha\|_2=1$ of
\begin{equation}
\label{eqn:step2}
\sum_{x<\N\kn\leq xe^{1/T}}~\sum_{(\mu_{\kp})_{\kp}\in\underline{\mu}[\kn]}\Big|\sum_{\pi\in\mathcal{S}(Q)}\Big[\prod_{\kp}s_{\mu_{\kp}}(A_{\pi}(\kp))\Big]\mathbf{1}_{(\kn,\kq_{\pi})}\alpha(\pi)\Big|^2.
\end{equation}
Let $\phi$ be a fixed smooth test function, supported in a compact subset of $[-2,2]$, such that $\phi(t)=1$ for $t\in[0,1]$ and $\phi(t)\in[0,1)$ otherwise.  Then \eqref{eqn:step2} is at most
\begin{equation}
\label{eqn:step3}
\sum_{\kn}\sum_{(\mu_{\kp})_{\kp}\in\underline{\mu}[\kn]}\Big|\sum_{\pi\in\mathcal{S}(Q)}\Big[\prod_{\kp}s_{\mu_{\kp}}(A_{\pi}(\kp))\Big]\mathbf{1}_{(\kn,\kq_{\pi})}\alpha(\pi)\Big|^2\phi\Big(T\log\frac{\N\kn}{x}\Big).
\end{equation}
We expand the square, interchange the order of summation, and apply \eqref{n-prime2S} to find that \eqref{eqn:step3} equals
\begin{equation}
\label{eqn:step4}
\begin{aligned}
&\sum_{\pi,\pi'\in\mathcal{S}(Q)}\alpha(\pi)\bar{\alpha(\pi')}\sum_{\kn}\sum_{(\mu_{\kp})_{\kp}\in\underline{\mu}[\kn]}\Big[\prod_{\kp}s_{\mu_{\kp}}(A_{\pi}(\kp))\Big]\overline{\Big[\prod_{\kp}s_{\mu_{\kp}}(A_{\pi'}(\kp))\Big]}\mathbf{1}_{(\kn,\kq_{\pi})}\mathbf{1}_{(\kn,\kq_{\pi'})}\phi\Big(T\log\frac{\N\kn}{x}\Big)\\
&=\sum_{\pi,\pi'\in\mathcal{S}(Q)}\alpha(\pi)\bar{\alpha(\pi')}\sum_{\gcd(\kn,\kq_{\pi}\kq_{\pi'})=\mathcal{O}_F }\lambda_{\pi\times\widetilde{\pi}'}(\kn)\phi\Big(T\log\frac{\N\kn}{x}\Big).
\end{aligned}
\end{equation}

Let $\kappa_{\pi\times\pi'}=\mathop{\mathrm{Res}}_{s=1}L(s,\pi\times\pi')\prod_{\kp \mid \kq_{\pi}\kq_{\pi'}}L(s,\pi_{\kp}\times\pi_{\kp}')^{-1}$. Note that $\kappa_{\pi\times\pi'}\geq 0$, with equality if and only if $\pi'\neq\tilde{\pi}$.  Since $|\alpha_{j,j',\pi\times\pi'}(\kp)|\leq \N\kp$, we have the bound
\[
\prod_{\kp \mid \kq_{\pi}\kq_{\pi'}}|L(1,\pi_{\kp}\times\pi_{\kp}')|^{-1}=\prod_{\kp \mid \kq_{\pi}\kq_{\pi'}}\prod_{j=1}^n \prod_{j'=1}^{n'}\Big|1-\frac{\alpha_{j,j',\pi\times\pi'}(\kp)}{\N\kp}\Big|\leq \prod_{\kp \mid\kq_{\pi}\kq_{\pi'}}2^{n'n}.
\]
Since $|\{\kp\colon \kp \mid \kn\}|\ll (\log\N\kn)/\log\log\N\kn$ \cite[Lemma 1.13b]{Weiss}, it follows from \eqref{eqn:Li_sharper} that
\begin{equation}
	\label{eqn:kappa_bound}
	\kappa_{\pi\times\tilde{\pi}}\ll_{\epsilon} C(\pi)^{\epsilon}.
\end{equation}

Let $\hat{\phi}(s)=\int_{\R}\phi(y)e^{sy}dy$.  It follows from a standard contour integral calculation using \eqref{eqn:LRS_2} and \eqref{eqn:preconvex} that \eqref{eqn:step4} equals
\begin{equation}
\label{eqn:cahrs}
\begin{aligned}
&\sum_{\pi,\pi'\in\mathcal{S}(Q)}\alpha(\pi)\bar{\alpha(\pi')}\Big(\frac{1}{2\pi iT}\int_{3-i\infty}^{3+i\infty}\frac{L(s,\pi\times\tilde{\pi}')}{\prod_{\kp \mid \kq_{\pi}\kq_{\pi'}}L(s,\pi_{\kp}\times\tilde{\pi}_{\kp}')}x^s\widehat{\phi}(s/T)ds\Big)\\
&=\sum_{\pi,\pi'\in\mathcal{S}(Q)}\alpha(\pi)\bar{\alpha(\pi')}\Big(\kappa_{\pi\times\tilde{\pi}'}x\frac{\widehat{\phi}(\frac{1}{T})}{T}+\frac{1}{2\pi iT}\int_{\frac{1}{4\log(ex)}-i\infty}^{\frac{1}{4\log(ex)}+i\infty}\frac{L(s,\pi\times\tilde{\pi}')x^s\widehat{\phi}(s/T)}{\prod_{\kp \mid \kq_{\pi}\kq_{\pi'}}L(s,\pi_{\kp}\times\tilde{\pi}_{\kp}')}ds\Big)\\
&=\sum_{\pi,\pi'\in\mathcal{S}(Q)}\alpha(\pi)\bar{\alpha(\pi')}\Big(\kappa_{\pi\times\tilde{\pi}'}x\frac{\widehat{\phi}(\frac{1}{T})}{T}+O_{\phi,\epsilon}(Q^{4n^2\theta_n+n+\epsilon}T^{\frac{1}{2}n^2 [F:\Q]+\epsilon})\Big).
\end{aligned}
\end{equation}
Recall that $\kappa_{\pi\times\tilde{\pi}'}>0$ when $\pi=\pi'$, and $\kappa_{\pi\times\tilde{\pi}'}=0$ otherwise.  Since $\|\alpha\|_2=1$ and $\phi$ is fixed, it follows from the inequality of arithmetic and geometric means that \eqref{eqn:cahrs} equals
\begin{multline*}
\frac{\widehat{\phi}(\frac{1}{T})}{T}x\sum_{\pi\in\mathcal{S}(Q)}|\alpha(\pi)|^2\kappa_{\pi\times\tilde{\pi}'}+O_{\phi}(Q^{4n^2\theta_n+n+\epsilon}T^{\frac{1}{2}n^2 [F:\Q]+\epsilon}|\mathcal{S}(Q)|)\\
\ll_{\phi,\epsilon}\frac{x}{T}\max_{\pi\in\mathcal{S}(Q)}\kappa_{\pi\times\tilde{\pi}'}+Q^{4n^2\theta_n+n+\epsilon}T^{\frac{1}{2}n^2 [F:\Q]+\epsilon}|\mathcal{S}(Q)|.
\end{multline*}
We estimate the maximum using \eqref{eqn:kappa_bound}, and the desired result follows.
\end{proof}

We use \cref{prop:pre_large_sieve} to prove \cref{thm:main_theorem_1}.

\begin{proof}[Proof of \cref{thm:main_theorem_1}]
For the sum involving $\mu_{\pi\times\pi'}(\kn)$, we apply \cref{prop:pre_large_sieve} with
\[
a((\mu_{\kp})_{\kp}) = b\Big(\prod_{\kp}\N\kp^{|\mu_{\kp}|}\Big)\mathbf{1}_{(\prod_{\kp}\N\kp^{|\mu_{\kp}|},\kq_{\pi'})}\prod_{\kp}s_{\mu_{\kp}^*}(-A_{\pi'}(\kp)).
\]
If $(\mu_{\kp})_{\kp}\in\underline{\mu}[\kn]$, then by \eqref{eqn:underline_mu_def}, we have that $|\mu_{\kp}|=\mathrm{ord}_{\kp}(\kn)$ and
\[
a((\mu_{\kp})_{\kp})=b(\kn)\mathbf{1}_{(\kn,\kq_{\pi'})}\prod_{\kp}s_{\mu_{\kp}^*}(-A_{\pi'}(\kp)).
\]
By \eqref{eqn:moebius}, the left-hand side of \cref{prop:pre_large_sieve} becomes
\begin{align*}
	&\sum_{\pi\in\mathcal{S}(Q)}\Big|\sum_{\substack{x<\N\kn\leq xe^{1/T} \\ \gcd(\kn,\kq_{\pi})=\mathcal{O}_F }}~\sum_{(\mu_{\kp})_{\kp}\in\underline{\mu}[\kn]}\Big[\prod_{\kp}s_{\mu_{\kp}}(A_{\pi}(\kp))\Big]a((\mu_{\kp})_{\kp})\Big|^2\\
	&=\sum_{\pi\in\mathcal{S}(Q)}\Big|\sum_{\substack{x<\N\kn\leq xe^{1/T} \\ \gcd(\kn,\kq_{\pi})=\mathcal{O}_F }}b(\kn)\mathbf{1}_{(\kn,\kq_{\pi'})}\sum_{(\mu_{\kp})_{\kp}\in\underline{\mu}[\kn]}\prod_{\kp}s_{\mu_{\kp}}(A_{\pi}(\kp))s_{\mu_{\kp}^*}(-A_{\pi'}(\kp)\Big|^2\\
	&=\sum_{\pi\in\mathcal{S}(Q)}\Big|\sum_{\substack{x<\N\kn\leq xe^{1/T} \\ \gcd(\kn,\kq_{\pi}\kq_{\pi'})=\mathcal{O}_F }}\mu_{\pi\times\pi'}(\kn)b(\kn)\Big|^2.
\end{align*}
Similarly, the right-hand side of \cref{prop:pre_large_sieve} becomes
\begin{align*}
\label{eqn:intermediary}
Q^{\epsilon}\Big(\frac{x}{T}+Q^{4n^2\theta_n+n}T^{\frac{1}{2}n^2 [F:\Q]+\epsilon}|\mathcal{S}(Q)|\Big)\sum_{\substack{\N\kn\in(x,xe^{1/T}] \\ \gcd(\kn,\kq_{\pi'})=\mathcal{O}_F }}|b(\kn)|^2\sum_{(\mu_{\kp})_{\kp}\in\underline{\mu}[\kn]}\Big|\prod_{\kp}s_{\mu_{\kp}^*}(-A_{\pi'}(\kp))\Big|^2.
\end{align*}
The desired result now follows from \eqref{eqn:dual_a}.  Apart from the new choice of
\[
a((\mu_{\kp})_{\kp}) = b\Big(\prod_{\kp}\N\kp^{|\mu_{\kp}|}\Big)\mathbf{1}_{(\prod_{\kp}\N\kp^{|\mu_{\kp}|},\kq_{\pi'})}\prod_{\kp}s_{\mu_{\kp}}(A_{\pi'}(\kp)),
\]
the result for the sum involving $\lambda_{\pi\times\pi'}(\kn)$ is handled in the same manner.
\end{proof}

\begin{corollary}
\label{cor:MVT}
	Let $\pi'\in\mathfrak{F}_{n'}$, $Q,T\geq 1$, and $\epsilon>0$.  If $Y\geq e$ and
	\[
	X\geq Q^{4n^2\theta_n+n+\epsilon}T^{\frac{1}{2}n^2 [F:\Q]+1+\epsilon}|\mathcal{S}(Q)|,\quad \log Y\asymp_{\epsilon} \log X,
	\]
	then
	\begin{align*}
	\sum_{\pi\in\mathcal{S}(Q)}\int_{-T}^{T}\Big|\sum_{\substack{\N\kn > X \\ \gcd(\kn,\kq_{\pi}\kq_{\pi'})=\mathcal{O}_F }}\frac{\mu_{\pi\times\pi'}(\kn)}{\N\kn^{1+\frac{1}{\log Y}+iv}}\Big|^2 dv &\ll_{\epsilon} C(\pi')^{\epsilon}Q^{\epsilon}\log X,\\
	\sum_{\pi\in\mathcal{S}(Q)}\int_{-T}^{T}\Big|\sum_{\substack{\N\kn\leq X \\ \gcd(\kn,\kq_{\pi}\kq_{\pi'})=\mathcal{O}_F }}\frac{\mu_{\pi\times\pi'}(\kn)}{\N\kn^{\frac{1}{2}+iv}}\Big|^2 dv&\ll_{\epsilon} C(\pi')^{\epsilon}Q^{\epsilon}X\log X.
	\end{align*}
\end{corollary}
\begin{proof}
	We prove the first result; the second result is proved completely analogously.  A formal generalization of \cite[Theorem 1]{Gallagher} to number fields tells us that if $c(\kn)$ is a complex-valued function supported on the integral ideals of $\cO_F$ with $\sum_{\kn}|c(\kn)|<\infty$, then
	\[
	\int_{-T}^{T}\Big|\sum_{\kn}\frac{c(\kn)}{\N\kn^{it}}\Big|^2 dt\ll T^2 \int_0^{\infty}\Big|\sum_{\N\kn\in(x,xe^{1/T}]}c(\kn)\Big|^2\frac{dx}{x}.
	\]
	We choose $b(\kn)=\N\kn^{-1-\frac{1}{\log Y}}$ if $\N\kn>X$ and $b(\kn)=0$ otherwise, which leads to
\[
\sum_{\pi\in\mathcal{S}(Q)}\int_{-T}^{T}\Big|\sum_{\substack{\N\kn > X \\ \gcd(\kn,\kq_{\pi}\kq_{\pi'})=\mathcal{O}_F }}\frac{\mu_{\pi\times\pi'}(\kn)}{\N\kn^{1+\frac{1}{\log Y}+it}}\Big|^2 dt\ll T^2 \int_0^{\infty}\sum_{\pi\in\mathcal{S}(Q)}\Big|\sum_{\N\kn\in(x,xe^{1/T}]}\mu_{\pi\times\pi'}(\kn)b(\kn)\Big|^2\frac{dx}{x}.
\]
	  \cref{thm:main_theorem_1} and the fact that $\lambda_{\pi'\times\tilde{\pi}'}(\kn)\geq 0$ for all $\kn$ imply that the above display is
	\begin{align*}
\ll_{\epsilon}  Q^{\epsilon}\sum_{\substack{\N\kn > X \\ \gcd(\kn,\kq_{\pi'})=\mathcal{O}_F }}\frac{\lambda_{\pi'\times\tilde{\pi}'}(\kn)}{\N\kn^{1+\frac{2}{\log Y}}}\Big(1+\frac{Q^{4n^2\theta_n+n}T^{\frac{1}{2}n^2 [F:\Q]+1+\epsilon}|\mathcal{S}(Q)|}{\N\kn}\Big)\ll_{\epsilon} Q^{\epsilon}\sum_{\N\kn > X}\frac{\lambda_{\pi'\times\tilde{\pi}'}(\kn)}{\N\kn^{1+\frac{2}{\log Y}}}.
	\end{align*}
	The desired result now follows from \cref{lem:mertens} and our choices of $X$ and $Y$.
\end{proof}

\section{Proof of \texorpdfstring{\cref{thm:ZDE}}{Theorem \ref*{thm:ZDE}}}
\label{sec:ZDE}

Let $n,n'\geq 1$, $\epsilon>0$, $T\geq 3$, $Q\geq 3$, $\pi\in\mathcal{S}(Q)$, and $\pi'\in\mathfrak{F}_{n'}$.  Define
\begin{equation}
\label{eqn:Xchoice}
\begin{aligned}
&X\coloneqq C(\pi')^{\epsilon}Q^{4n^2\theta_n+n+\epsilon}T^{\frac{1}{2}n^2 [F:\Q]+1+\epsilon}|\mathcal{S}(Q)|,\\
&Y\coloneqq  \big((C(\pi')^{\frac{n}{2}}Q^{\frac{n'}{2}} (C(\pi')Q)^{2n'n\max\{0,\theta_n+\theta_{n'}-\frac{1}{2}\}} |\mathcal{S}(Q)| T^{\frac{n'n[F:\Q]}{2}+1})^{1+\frac{\epsilon}{4n'n[F:\Q]}}X\big)^{\frac{1}{3-2\sigma}},\\
&L^{\mathrm{ur}}(s,\pi\times\pi')\coloneqq L(s,\pi\times\pi')\prod_{\kp \mid \kq_{\pi}\kq_{\pi'}}L(s,\pi_{\kp}\times\pi_{\kp}')^{-1},\\
&M_X(s,\pi\times\pi')\coloneqq \sum_{\substack{\N\kn\leq X \\ \gcd(\kn,\kq_{\pi}\kq_{\pi'})=\mathcal{O}_F }}\frac{\mu_{\pi\times\pi'}(\kn)}{\N\kn^s},\\
&LM_X(s,\pi\times\pi')\coloneqq L^{\mathrm{ur}}(s,\pi\times\pi') M_X(s,\pi\times\pi').
\end{aligned}
\end{equation}
Note that $\log(C(\pi)C(\pi')T)\ll \log X\asymp_{\epsilon} \log Y$.  We assume that $T\geq 2(\log Y)^2$ since the proof does not change appreciably otherwise.  In this section, $\epsilon$ might vary from line to line, and terms of size $(C(\pi')QT^{[F:\Q]}|\mathcal{S}(Q)|)^{\epsilon}$ and polynomials in $\log(C(\pi')QT^{[F:\Q]}|\mathcal{S}(Q)|)$ might be bounded by $X^{\epsilon}$ without mention.

If $t\in[-T,T]$, then
\[
|\{\rho=\beta+i\gamma\colon \beta\geq 0,~|\gamma-t|\leq 1\}|\ll \log C(\pi\times\pi',t)\ll \log(C(\pi)C(\pi')T)
\]
by \cite[Proposition 5.7]{IK}.  We observe that the rectangle $[\sigma,1]\times[-T,T]$ is covered by $O(T)$ boxes of the form $[\sigma,1]\times[y,y+2(\log Y)^2]$, each containing $O((\log(C(\pi)C(\pi')T))^3)$ zeros.  If we write $n_{\pi\times\pi'}$ for the number of such boxes containing at least one zero $\rho$ of $L(s,\pi\times\pi')$, then
\[
N_{\pi\times\pi'}(\sigma,T)\ll (\log(C(\pi')Q T))^3 n_{\pi\times\pi'}.
\]
If $\pi'=\tilde{\pi}$, then it suffices to assume that $|\rho-1|>1$ since there are $O(\log(C(\pi')QT))$ zeros $\rho$ such that $|\rho-1|\leq 1$.

Let $\rho=\beta+i\gamma$ be a zero of $L(s,\pi\times\pi')$, in which case $LM_X(\rho,\pi\times\pi')=0$.  We compute
\begin{equation}
\begin{aligned}
\label{eqn:zero_detect_1}
e^{-1/Y}&=\frac{1}{2\pi i}\int_{1-\beta+\frac{1}{\log Y}-i\infty}^{1-\beta+\frac{1}{\log Y}+i\infty}(1-LM_X(\rho+w,\pi\times\pi')\Gamma(w)Y^{w}dw\\
&+\frac{1}{2\pi i}\int_{\frac{1}{2}-\beta-i\infty}^{\frac{1}{2}-\beta+i\infty}LM_X(\rho+w,\pi\times\pi')\Gamma(w)Y^{w}dw\\
&+\kappa_{\pi\times\pi'}\Gamma(1-\rho)Y^{1-\rho}\sum_{\substack{\N\kn\leq X \\ \gcd(\kn,\kq_{\pi}\kq_{\pi'})=\cO_F}}\frac{\mu_{\pi\times\pi'}(\kn)}{\N\kn}.
\end{aligned}
\hspace{-.099in}
\end{equation}
It follows from \eqref{eqn:preconvex} and \cref{lem:mertens,lem:dual_partition} that if $\re(w) = 1-\beta+\frac{1}{\log Y}$, then
\begin{equation}
\label{eqn:tail_est}
\begin{aligned}
	&|1-LM_X(\rho+w,\pi\times\pi')|\\
	&= |L^{\mathrm{ur}}(\rho+w,\pi\times\pi')|\cdot|L^{\mathrm{ur}}(\rho+w,\pi\times\pi')^{-1}-M_X(\rho+w,\pi\times\pi')|\\
	&\ll_{\epsilon} C(\pi\times\pi',|\gamma+\im(w)|+1)^{\epsilon}\sum_{\substack{\N\kn>X \\ \gcd(\kn,\kq_{\pi}\kq_{\pi'})=\mathcal{O}_F }}\frac{|\mu_{\pi\times\pi'}(\kn)|}{\N\kn^{1+\frac{1}{\log Y}}}\\
	&\ll_{\epsilon} C(\pi\times\pi',|\gamma+\im(w)|+1)^{\epsilon}\sum_{\substack{\N\kn>X \\ \gcd(\kn,\kq_{\pi}\kq_{\pi'})=\mathcal{O}_F }}\frac{\lambda_{\pi\times\tilde{\pi}}(\kn)+\lambda_{\pi'\times\tilde{\pi}'}(\kn)}{\N\kn^{1+\frac{1}{\log Y}}}\\
	&\ll_{\epsilon} C(\pi)^{\epsilon}C(\pi')^{\epsilon}(|\gamma+\im(w)|+1)^{\epsilon}\log X.
\end{aligned}
\end{equation}
Therefore, we deduce from Stirling's formula that
\begin{multline*}
	\Big|\frac{1}{2\pi i}\int_{1-\beta+\frac{1}{\log Y}-i\infty}^{1-\beta+\frac{1}{\log Y}+i\infty}(1-LM_X(\rho+w,\pi\times\pi'))\Gamma(w)Y^{w}dw\\
	-\frac{1}{2\pi i}\int_{1-\beta+\frac{1}{\log Y}-i(\log Y)^2}^{1-\beta+\frac{1}{\log Y}+i(\log Y)^2}(1-LM_X(\rho+w,\pi\times\pi'))\Gamma(w)Y^{w}dw\Big|\ll \frac{1}{Y}.
\end{multline*}
The other terms in \eqref{eqn:zero_detect_1} are handled similarly.  Since $e^{-1/Y}=1+O(1/Y)$, it follows that
\begin{equation}
\begin{aligned}
\label{eqn:zero_detect_2}
1&\ll \frac{1}{2\pi i}\int_{1-\beta+\frac{1}{\log Y}-i(\log Y)^2}^{1-\beta+\frac{1}{\log Y}+i(\log Y)^2}(1-LM_X(\rho+w,\pi\times\pi'))\Gamma(w)Y^{w}dw\\
&+\frac{1}{2\pi i}\int_{\frac{1}{2}-\beta-i(\log Y)^2}^{\frac{1}{2}-\beta+i(\log Y)^2}LM_X(\rho+w,\pi\times\pi')\Gamma(w)Y^{w}dw+\kappa_{\pi\times\pi'}Y^{1-\sigma}\frac{\log X}{\max\{1,|\gamma|^3\}}.
\end{aligned}
\end{equation}

Write the first integral in \eqref{eqn:zero_detect_2} as $I_1$ and the second as $I_2$.   A simple optimization calculation shows that if $c\geq 1$ and $c^{-1}\leq |I_1|+|I_2|$, then $c^{-1}\leq 2c(|I_1|^2+|I_2|)$.  We estimate $|I_1|^2$ using the Cauchy--Schwarz inequality, thus obtaining the bound
\begin{align*}
1&\ll (\log Y)^2 Y^{2(1-\sigma)}\int_{\gamma-(\log Y)^2}^{\gamma+(\log Y)^2}|1-LM_X(1+\tfrac{1}{\log Y}+iv,\pi\times\pi')|^2 dv\\
&+Y^{\frac{1}{2}-\sigma}\int_{\gamma-(\log Y)^2}^{\gamma+(\log Y)^2}|LM_X(\tfrac{1}{2}+iv,\pi\times\pi')|dv+\kappa_{\pi\times\pi'}Y^{1-\sigma}\frac{\log X}{\max\{1,|\gamma|^3\}}.
\end{align*}
Since $T\geq 2(\log Y)^2$ by hypothesis, we conclude that
\begin{equation}
\label{eqn:zero_detect_3}
\begin{aligned}
n_{\pi\times\pi'}&\ll (\log Y)^2 Y^{2(1-\sigma)}\int_{-T}^{T}|1-L(1+\tfrac{1}{\log Y}+iv)M_X(1+\tfrac{1}{\log Y}+iv)|^2 dv\\
&+Y^{\frac{1}{2}-\sigma}\int_{-T}^{T}|L(\tfrac{1}{2}+iv)M_X(\tfrac{1}{2}+iv)|dv+\kappa_{\pi\times\pi'}Y^{1-\sigma}\log X+1.
\end{aligned}
\end{equation}

Note that $\kappa_{\pi\times\pi'}=0$ unless $\pi'=\tilde{\pi}$, which occurs for at most one $\pi\in\mathcal{S}(Q)$.  When $\kappa_{\pi\times\pi'}\neq 0$, it satisfies $\kappa_{\pi\times\pi'}=\kappa_{\pi\times\tilde{\pi}}\ll_{\epsilon} C(\pi)^{\epsilon}$ per \eqref{eqn:kappa_bound}.  Therefore, summing \eqref{eqn:zero_detect_3} over $\pi\in\mathcal{S}(Q)$, we find that
\begin{align}
\label{eqn:bound_int}
\sum_{\pi\in\mathcal{S}(Q)}N_{\pi\times\pi'}(\sigma,T)&\ll (\log C(\pi)C(\pi')T)^3\sum_{\pi\in\mathcal{S}(Q)}n_{\pi\times\pi'}\notag\\
&\ll_{\epsilon} X^{\epsilon} \Big(Y^{2(1-\sigma)}\int_{-T}^{T}\sum_{\pi\in\mathcal{S}(Q)}|1-LM_X(1+\tfrac{1}{\log Y}+iv,\pi\times\pi')|^2 dv\\
&\qquad\qquad+Y^{\frac{1}{2}-\sigma}\int_{-T}^{T}\sum_{\pi\in\mathcal{S}(Q)}|LM_X(\tfrac{1}{2}+iv,\pi\times\pi')|dv+Y^{1-\sigma}\Big).\notag
\end{align}
For the first integral in \eqref{eqn:bound_int}, we deduce from \eqref{eqn:preconvex} that
\begin{align*}
	&|1-LM_X(1+\tfrac{1}{\log Y}+it,\pi\times\pi')|^2\\
	&=|L^{\mathrm{ur}}(1+\tfrac{1}{\log Y}+it,\pi\times\pi')|^2\Big|\frac{1}{L^{\mathrm{ur}}(1+\tfrac{1}{\log Y}+it,\pi\times\pi')}-M_X(1+\tfrac{1}{\log Y}+it,\pi\times\pi')\Big|^2\\
	&\ll_{\epsilon} X^{\epsilon}\Big|\sum_{\substack{\N\kn>X \\ \gcd(\kn,\kq_{\pi}\kq_{\pi'})=\cO_F}}\frac{\mu_{\pi\times\pi'}(\kn)}{\N\kn^{1+\frac{1}{\log Y}}}\Big|^2.
\end{align*}
For the second integral in \eqref{eqn:bound_int}, it follows from the Cauchy--Schwarz inequality that
\begin{align*}
&\int_{-T}^{T}\sum_{\pi\in\mathcal{S}(Q)}|LM_X(\tfrac{1}{2}+iv,\pi\times\pi')|dv\\
&\leq \Big(\int_{-T}^{T}\sum_{\pi\in\mathcal{S}(Q)}|L^{\mathrm{ur}}(\tfrac{1}{2}+iv,\pi\times\pi')|^2 dv\Big)^{\frac{1}{2}}\Big(\int_{-T}^{T}\sum_{\pi\in\mathcal{S}(Q)}\Big|\sum_{\substack{\N\kn\leq X\\ \gcd(\kn,\kq_{\pi}\kq_{\pi'})=\cO_F}}\frac{\mu_{\pi\times\pi'}(\kn)}{\N\kn^{\frac{1}{2}+iv}}\Big|^2 dv\Big)^{\frac{1}{2}}.
\end{align*}
We bound the second moment of $L^{\mathrm{ur}}(\frac{1}{2}+iv,\pi\times\pi')$ trivially as
\begin{align*}
&\int_{-T}^{T}\sum_{\pi\in\mathcal{S}(Q)}|L^{\mathrm{ur}}(\tfrac{1}{2}+iv,\pi\times\pi')|^2 dv \\
&\ll_{\epsilon} (QC(\pi'))^{2n'n\max\{0,\theta_n+\theta_{n'}-\frac{1}{2}\}+\epsilon} Q^{\frac{n'}{2}}C(\pi')^{\frac{n}{2}}T^{\frac{n'n[F:\Q]}{2}+1+\epsilon}|\mathcal{S}(Q)| \ll_{\epsilon}\frac{Y^{3-2\sigma}}{X}
\end{align*}
using \eqref{eqn:LRS_2} (to bound the ramified Euler factors) and \eqref{eqn:preconvex}.  In summary, we find that
\begin{align*}
\sum_{\pi\in\mathcal{S}(Q)}N_{\pi\times\pi'}(\sigma,T)&\ll_{\epsilon} X^{\epsilon} \Big(Y^{2(1-\sigma)}\int_{-T}^{T}\sum_{\pi\in\mathcal{S}(Q)}\Big|\sum_{\substack{\N\kn>X \\ \gcd(\kn,\kq_{\pi}\kq_{\pi'})=\cO_F}}\frac{\mu_{\pi\times\pi'}(\kn)}{\N\kn^{1+\frac{1}{\log Y}+iv}}\Big|^2 dv\\
&+ Y^{\frac{1}{2}-\sigma}\Big(\frac{Y^{3-2\sigma}}{X}\int_{-T}^{T}\sum_{\pi\in\mathcal{S}(Q)}\Big|\sum_{\substack{\N\kn\leq X\\ \gcd(\kn,\kq_{\pi}\kq_{\pi'})=\cO_F}}\frac{\mu_{\pi\times\pi'}(\kn)}{\N\kn^{\frac{1}{2}+iv}}\Big|^2 dv\Big)^{\frac{1}{2}}+ Y^{1-\sigma}\Big).
\end{align*}
We bound the two $v$-integrals using \cref{cor:MVT} and \eqref{eqn:Xchoice}, thus obtaining
\[
\sum_{\pi\in\mathcal{S}(Q)}N_{\pi\times\pi'}(\sigma,T)\ll_{\epsilon} Y^{2(1-\sigma)}X^{\epsilon}.
\]
\cref{thm:ZDE} now follows from \eqref{eqn:Xchoice}, the bounds for $\theta_{n}$ and $\theta_{n'}$ in  \eqref{eqn:ramanujan_progress}, and \eqref{eqn:poly_upper}.

\section{Bounds for Rankin--Selberg $L$-functions}
\label{sec:good_ZFR}

In this section, we prove \cref{thm:subconvexity}.  Let $\pi'\in\mathfrak{F}_{n'}$ and $Q\geq 1$.  In \eqref{eqn:ZDE_F}, we rescale $\epsilon$ to $\epsilon/72$ and define $\alpha = \epsilon/(7.2\max\{n,n'\})$ so that
\begin{equation}
\label{eqn:thm_specialized}
|\{\pi\in\mathfrak{F}_n(Q)\colon N_{\pi\times\pi'}(1-\alpha,6)=0\}|\leq \sum_{\pi\in\mathfrak{F}_n(Q)}N_{\pi\times\pi'}(1-\alpha,6)\ll_{F,\epsilon}(C(\pi')^n |\mathfrak{F}_n(Q)|)^{\epsilon}.
\end{equation}
Let $\pi\in\mathfrak{F}_n(Q)$.  Proceeding as in the proof of \cite[Theorem 1.1]{ST}, we obtain\footnote{Small modifications are required to work over $F\neq\Q$, none of which substantially changes the proof.}
\[
\log|L(\tfrac{1}{2},\pi\times\pi')|\leq\Big(\frac{1}{4}-\frac{\alpha}{10^{9}}\Big)\log C(\pi\times\pi') +2\log|L(\tfrac{3}{2},\pi\times\pi')|+ \frac{\alpha}{10^{7}} N_{\pi\times\pi'}(1-\alpha,6)+O(1).
\]
By \eqref{eqn:Li_sharper} and \eqref{eqn:preconvex}, we find that there exists an effectively computable constant $\Cl[abcon]{kappa_const}=\Cr{kappa_const}(n,n',[F:\Q],\epsilon)>0$ such that if $C(\pi\times\pi')\geq \Cr{kappa_const}$, then
\begin{equation}
\label{eqn:ST_finale}
\log|L(\tfrac{1}{2},\pi\times\pi')|\leq\Big(\frac{1}{4}-\frac{9\alpha}{10^{10}}\Big)\log C(\pi\times\pi')\\
+\frac{\alpha}{10^7}N_{\pi\times\pi'}(1-\alpha,6).
\end{equation}
\cref{thm:subconvexity} now follows from \eqref{eqn:thm_specialized} and \eqref{eqn:ST_finale}.

\section{Effective multiplicity one on average}
\label{sec:multiplicity_one}

In this section, we prove \cref{thm:multiplicity_one}.  We consider the family of cuspidal automorphic representations $\mathfrak{F}_n(Q)$ over a number field $F$.  Recall our convention that implied constants are allowed to depend on $n\geq 3$, $[F:\Q]$, and $\epsilon$ unless specifically mentioned otherwise.

To prove \cref{thm:multiplicity_one}, we use \cref{thm:ZDE} to build large zero-free regions for almost all $L$-functions $L(s,\pi\times\pi')$ with $\pi\in\mathfrak{F}_n$ varying.  If $\pi'\in\{\tilde{\pi},\tilde{\pi}'\}$, then $L(s,\pi\times\pi')$ has the standard zero-free region \eqref{eqn:std_ZFR} apart from a possible Landau--Siegel zero, which must be both real and simple.  In all other cases, only Brumley's much narrower zero-free region is known (\cite{Brumley}, \cite[Appendix]{Lapid}).  We can use \cref{thm:ZDE} to establish a much stronger zero-free region for $L(s,\pi\times\pi')$ for all $\pi\in\mathfrak{F}_n(Q)$ with very few exceptions.  Here is a simple example.
\begin{corollary}
	\label{cor:good_ZFR}
	Let $\epsilon>0$, $n\geq 3$, and $\pi'\in\mathfrak{F}_n(Q)$.  For all $\pi\in\mathfrak{F}_n(Q)$ with $O_{\epsilon}(Q^{\epsilon})$ exceptions, the $L$-function $L(s,\pi\times\pi')$ is nonzero in the region
\[
\re(s)\geq 1-\frac{\epsilon}{20n^2},\qquad |\im(s)|\leq \log Q.
\]
\end{corollary}

\begin{proof}
This follows from \eqref{eqn:ZDE_Q} (once we rescale $\epsilon$ to $\epsilon/21$) with $C(\pi')\leq Q$, $T = \log Q$, and $\sigma = 1-\epsilon/(20n^2)$.
\end{proof}

For $\pi_1,\pi_2\in\mathfrak{F}_n$ define the numbers $\Lambda_{\pi_1\times\pi_2}(\kn)$ by the Dirichlet series identity (for $\re(s)>1$)
\begin{equation}
\label{eqn:vonMangdef}
\sum_{\kn}\frac{\Lambda_{\pi_1\times\pi_2}(\kn)}{\N\kn^s}= -\frac{L'}{L}(s,\pi_1\times\pi_2)= \sum_{\kp} \sum_{k = 1}^{\infty} \frac{\sum_{j=1}^{n}\sum_{j'=1}^{n}\alpha_{j,j',\pi_1\times\pi_2}(\kp)^k\log\N\kp}{\N\kp^{ks}}.
\end{equation}

\begin{lemma}
\label{lem:HT}
	\label{lem:HumphriesT}
	Let $\pi'\in\mathfrak{F}_n$.  There exist absolute and effectively computable constants $\Cl[abcon]{Hoh5}\in(0,1)$; $\Cl[abcon]{Hoh6},\Cl[abcon]{Hoh1},\Cl[abcon]{Hoh2},\Cl[abcon]{Hoh3}\geq 1$; and $\Cl[abcon]{Hoh4}\in(0,1)$ such that the following are true.
	\begin{enumerate}
		\item The $L$-function $L(s,\pi'\times\tilde{\pi}')$ has at most one zero, say $\beta_1$, in the region
		\[
		\re(s)\geq 1- \frac{\Cr{Hoh5}}{\log(C(\pi')^n(|\im(s)|+e)^{n^2[F:\Q]})}.
		\]
		If $\beta_1$ exists, then it must be real and simple and satisfy $\beta_1\leq 1-C(\pi')^{-\Cr{Hoh6}n}$.
		\item If $A\geq \Cr{Hoh1}$, $\log\log C(\pi')\geq \Cr{Hoh2}n^4[F:\Q]^2$, and $x\geq C(\pi')^{\Cr{Hoh3}A^2n^3[F:\Q]\log (en[F:\Q])}$, then
	\[
	\sum_{\frac{x}{2}<\N\kn\leq x}\Lambda_{\pi'\times\tilde{\pi}'}(\kn)=\begin{cases}
		\frac{x}{2}(1-\xi^{\beta_1-1})(1+O(e^{-\Cr{Hoh4}A}))&\mbox{if $\beta_1$ exists,}\\
		\frac{x}{2}(1+O(e^{-\Cr{Hoh4}A}))&\mbox{otherwise,}
	\end{cases}
	\]
	where $\xi\in[\frac{x}{2},x]$ satisfies $x^{\beta_1}-(\frac{x}{2})^{\beta_1}=\beta_1 \frac{x}{2}\xi^{\beta_1-1}$ and the implied constants are absolute and effectively computable.
	\end{enumerate}
\end{lemma}
\begin{proof}
This is \cite[Theorem 2.1]{HT} with $\delta=0$ and $x$ replaced with $x/2$.
\end{proof}

\cref{lem:HT} informs the following choices of parameters that we will use throughout the rest of this section.  Let $0<\epsilon<1$, let $Q$ be sufficiently large with respect to $\epsilon$, and let
\begin{equation}
\begin{aligned}
\label{eqn:params_M1}
&\Cr{Hoh6}\geq \frac{\Cr{Hoh1}^2\Cr{Hoh3}}{3240},\qquad A = \sqrt{\frac{3240 \Cr{Hoh6}}{\Cr{Hoh3}\epsilon}},\qquad  B= \frac{41n^2}{\epsilon},\qquad x=\frac{1}{2}(\log Q)^B,\\
&\pi\in\mathfrak{F}_n(Q),\qquad \pi'\in \mathfrak{F}_n\big(x^{\frac{1}{A^2\Cr{Hoh3}n^4[F:\Q]^2}}\big)=\mathfrak{F}_n\big((2^{-\frac{\epsilon}{41n^2}}\log Q)^{\frac{41}{3240\Cr{Hoh6}n^2[F:\Q]^2}}\big).
\end{aligned}
\end{equation}

\begin{corollary}
	\label{cor:lower_S}
	Under the notation and hypotheses of \eqref{eqn:params_M1}, there holds
	\[
	\sum_{\substack{\frac{x}{2}<\N\kn\leq x \\ \gcd(\kn,\kq_{\pi}\kq_{\pi'})=\mathcal{O}_F }}\Lambda_{\pi'\times\tilde{\pi}'}(\kn)\gg_{\epsilon}  x^{1-\frac{2\Cr{Hoh6}}{A^2\Cr{Hoh3}n^3[F:\Q]^2}}.
	\]
\end{corollary}
\begin{proof}
Let $\pi,\pi'\in\mathfrak{F}_n(Q)$.  In \cref{lem:HumphriesT}, the lower bound on $C(\pi')$ only serves to ensure that the implied constants are absolute.  This was pertinent in \cite{HT}, but it is not pertinent here.  Thus, we may replace the two conditions $\log\log C(\pi')\geq \Cr{Hoh2}n^4[F:\Q]^2$ and $x\geq C(\pi')^{A^2 \Cr{Hoh3}n^3[F:\Q]\log (en[F:\Q])}$ with the single condition $C(\pi')\ll x^{1/(A^2\Cr{Hoh3}n^4[F:\Q]^2)}$.  We want to refine \cref{lem:HumphriesT} so that one only sums over $\kn$ such that $\gcd(\kn,\kq_{\pi}\kq_{\pi'})=\mathcal{O}_F $.  Using \eqref{eqn:LRS_finite} and \eqref{eqn:LRS_2}, we find that the contribution to \cref{lem:HumphriesT}(2) arising from $\kn$ such that $\gcd(\kn,\kq_{\pi}\kq_{\pi'})\neq\cO_F$ is
\begin{align*}
\sum_{\substack{\frac{x}{2}<\N\kn\leq x \\ \gcd(\kn,\kq_{\pi}\kq_{\pi'})\neq\cO_F}}\Lambda_{\pi'\times\tilde{\pi}'}(\kn)\ll  \sum_{\kp \mid \kq_{\pi}\kq_{\pi'}} \sum_{\frac{\log(x/2)}{\log\N\kp}<j\leq \frac{\log x}{\log\N\kp}}\N\kp^{j(1-\frac{2}{n^2+1})}\ll \frac{\log Q}{\log\log Q}x^{1-\frac{2}{n^2+1}},
\end{align*}
which is $\ll x^{1-\frac{1}{n^2+1}}$.  In the worst case, where $\beta_1$ in \cref{lem:HumphriesT}(1) exists, it follows from \cref{lem:HumphriesT}(2) and the above discussion that if $C(\pi')\ll_{n,[F:\Q]} x^{1/(A^2\Cr{Hoh3}n^3[F:\Q]\log(en[F:\Q]))}$, then
\begin{align*}
	\sum_{\substack{x<\N\kn\leq 2x \\ \gcd(\kn,\kq_{\pi}\kq_{\pi'})=\mathcal{O}_F }}\Lambda_{\pi'\times\tilde{\pi}'}(\kn)\gg x(1-\xi^{\beta_1-1})\gg_{n,[F:\Q]} x\min\{1,(1-\beta_1)\log C(\pi')\}\gg \frac{x\log C(\pi')}{C(\pi')^{\Cr{Hoh6}n}}.
\end{align*}
The desired result now follows.
\end{proof}

\begin{lemma}
\label{lem:upper_bound_S}
	Let $n\geq 3$.  Assume the notation and hypotheses in \eqref{eqn:params_M1}.  Let $\Phi$ be a fixed smooth function supported on a compact subset of $[\frac{1}{4},2]$ such that $0\leq \Phi(t)\leq 1$ for all $t\in[\frac{1}{4},2]$ and $\Phi(t)=1$ for $t\in(\frac{1}{2},1]$.  If $L(s,\pi\times\pi')$ is entire and does not vanish in the region $\re(s)\geq 1-\epsilon/(20n^2)$ and $|\im(s)|\leq \log Q$, then
	\[
	\sum_{\gcd(\kn,\kq_{\pi}\kq_{\pi'})=\mathcal{O}_F }\Lambda_{\pi\times\tilde{\pi}'}(\kn)\Phi(\N\kn/x)\ll_{B} x^{1+\frac{2}{B}-\frac{\epsilon}{20n^2}}.
	\]
\end{lemma}
\begin{proof}
Writing $\hat{\Phi}(s)=\int_0^{\infty}\Phi(t)t^{s-1}dt$ for the Mellin transform of $\Phi$ (which is an entire function of $s$), we observe via Mellin inversion that
\[
\sum_{\gcd(\kn,\kq_{\pi}\kq_{\pi'})=\mathcal{O}_F }\Lambda_{\pi\times\tilde{\pi}'}(\kn)\Phi(\N\kn/x) = \frac{1}{2\pi i}\int_{2-i\infty}^{2+i\infty}-\frac{(L^{\textup{ur}})'}{(L^{\textup{ur}})}(s,\pi\times\tilde{\pi}')\hat{\Phi}(s)x^s ds,
\]
where $L^{\textup{ur}}(s,\pi\times\tilde{\pi}')=L(s,\pi\times\tilde{\pi}')\prod_{\kp\mid\kq_{\pi}\kq_{\pi'}}L(s,\pi_{\kp}\times\tilde{\pi}_{\kp}')^{-1}$.  A standard contour integral calculation using the argument principle shows that the above display equals
\[
-\sum_{L^{\textup{ur}}(\rho,\pi\times\tilde{\pi}')=0}\hat{\Phi}(\rho)x^{\rho},
\]
where $\rho$ ranges over all zeros of $L^{\textup{ur}}(s,\pi\times\tilde{\pi}')$.

Since $\Phi$ is compactly supported and $\hat{\Phi}$ is entire, it follows that for any $R\geq 2$, we have $|\hat{\Phi}(s)|\ll_{R}\min\{1,|s|^{-R}e^{\re(s)}\}$.  Note that the reciprocals of the Euler factors of $L(s,\pi\times\tilde{\pi}')$ at prime ideals $\kp \mid \kq_{\pi}\kq_{\pi'}$ and all of the trivial zeros of $L(s,\pi\times\tilde{\pi}')$ have real part no larger than $1-\frac{2}{n^2+1}$ per \eqref{eqn:LRS_finite} and \eqref{eqn:LRS_2}.  Since for any $t\in\R$ there are $\ll\log Q+\log(|t|+2)$ zeros $\rho=\beta+i\gamma$ of $L^{\textup{ur}}(s,\pi\times\tilde{\pi}')$ that satisfy $0<\beta<1$ and $|\gamma-t|\leq 1$, we find for any $T\geq 1$, $R\geq 2$, and $\sigma_0\geq 1-\frac{2}{n^2+1}$ that
\begin{equation}
\label{eqn:zero_empty_sum}
\sum_{L^{\textup{ur}}(\rho,\pi\times\tilde{\pi}')=0}\hat{\Phi}(\rho)x^{\rho}=\sum_{\substack{L^{\textup{ur}}(\rho,\pi\times\tilde{\pi}')=0 \\ \beta\geq \sigma_0 \\ |\gamma|\leq T}}\hat{\Phi}(\rho)x^{\rho}+O(T^{1-R}\log(QT)x+T\log(QT)x^{\sigma_0}).
\end{equation}
We choose $T=\log Q=(2x)^{1/B}$ and $\sigma_0=1-\epsilon/(20n^2)$, in which case our hypotheses imply that the sum over zeros on the right-hand side of \eqref{eqn:zero_empty_sum} is empty and
\[
\Big|\sum_{L^{\textup{ur}}(\rho,\pi\times\tilde{\pi}')=0}\hat{\Phi}(\rho)x^{\rho}\Big|\ll_{B} x^{1+\frac{2-R}{B}}+x^{\frac{2}{B}+\sigma_0}.
\]
The desired result follows from choosing $R=\max\{B(1-\sigma_0),3\}$.
\end{proof}

\begin{proof}[Proof of \cref{thm:multiplicity_one}]

Suppose to the contrary that $\pi\neq\pi'$ and $\pi_{\kp}\cong\pi_{\kp}'$ for all $\kp\nmid\kq_{\pi}\kq_{\pi'}$ with $\N\kp\leq 2x$.  Then $A_{\pi}(\kp)=A_{\pi'}(\kp)$ for all $\kp\nmid\kq_{\pi}\kq_{\pi'}$ with $\N\kp\leq 2x$.  By \eqref{eqn:separate_dirichlet_coeffs} and \eqref{eqn:vonMangdef}, it follows that $\Lambda_{\pi\times\tilde{\pi}'}(\kn)=\Lambda_{\pi'\times\tilde{\pi}'}(\kn)$ for all $\kn$ such that $\gcd(\kn,\kq_{\pi}\kq_{\pi'})=\mathcal{O}_F $ and $\N\kn\leq x$.  Since $\Lambda_{\pi'\times\tilde{\pi}'}(\kn)\geq 0$ for all such $\kn$, the same holds for $\Lambda_{\pi\times\tilde{\pi}'}(\kn)$.  It follows that
\[
\sum_{\substack{\frac{x}{2}<\N\kn\leq x \\ \gcd(\kn,\kq_{\pi}\kq_{\pi'})=\mathcal{O}_F }}\Lambda_{\pi'\times\tilde{\pi}'}(\kn)=\sum_{\substack{\frac{x}{2}<\N\kn\leq x \\ \gcd(\kn,\kq_{\pi}\kq_{\pi'})=\mathcal{O}_F }}\Lambda_{\pi\times\tilde{\pi}'}(\kn).
\]

By \eqref{eqn:params_M1}, we have that
\begin{equation}
\label{eqn:constraint_c_prime}
C(\pi')\ll x^{1/(A^2\Cr{Hoh3}n^4[F:\Q]^2)}.
\end{equation}
Therefore, if $L(s,\pi\times\tilde{\pi}')\neq 0$ in the region $\{s\in\mathbb{C}\colon \Re(s)\geq 1-\epsilon/(20n^2),~|\im(s)|\leq \log Q\}$, then by \cref{cor:lower_S} and \cref{lem:upper_bound_S}, we have that
\[
x^{1-\frac{2\Cr{Hoh6}}{A^2\Cr{Hoh3}n^3[F:\Q]^2}}\ll \sum_{\substack{\frac{x}{2}<\N\kn\leq x \\ \gcd(\kn,\kq_{\pi}\kq_{\pi'})=\mathcal{O}_F }}\Lambda_{\pi'\times\tilde{\pi}'}(\kn)=\sum_{\substack{\frac{x}{2}<\N\kn\leq x \\ \gcd(\kn,\kq_{\pi}\kq_{\pi'})=\mathcal{O}_F }}\Lambda_{\pi\times\tilde{\pi}'}(\kn)\ll_{B}x^{1+\frac{2}{B}-\frac{\epsilon}{20n^2}}.
\]
This implies that
\[
1-\frac{2\Cr{Hoh6}}{A^2\Cr{Hoh3}n^3[F:\Q]^2}\leq 1+\frac{2}{B}-\frac{\epsilon}{20n^2},
\]
which contradicts our choices of $A$ and $B$ in \eqref{eqn:params_M1}.  The desired result now follows from \cref{cor:good_ZFR}.
\end{proof}

\section{Automorphic level of distribution}
\label{sec:BV}

In this section, we prove \cref{thm:BV}.  In what follows, let $F=\Q$.  If $n=2$, then $L(s,\pi\times\tilde{\pi})$ is the $L$-function of an isobaric automorphic representation of $\GL_4(\A_{\Q})$ whose cuspidal constituents have rank at most 3.  Since the $L$-function any Dirichlet character twist of any  cuspidal constituents of rank 2 or 3 have no Landau--Siegel zero by \cite{Banks,Hoffstein}, a stronger result than \cref{thm:BV} follows from a minor variation of the proof of \cite[Theorem 1.1]{MR4530115}.  Therefore, we may restrict our consideration to $n\geq 3$.

Let $\pi\in\mathfrak{F}_n$ have arithmetic conductor $q_{\pi}$, and let $\chi$ be a primitive Dirichlet character modulo $q$.  We will allow $\pi$ to be fixed, so for notational compactness, we introduce $L_{\chi}(s) \coloneqq  L(s,\pi\times(\tilde{\pi}\otimes\chi))$ and $L_1(s)\coloneqq L(s,\pi\times\tilde{\pi})$. If $q$ and $q_{\pi}$ are coprime, then $L_{\chi}(s)$ is entire if and only if $\chi$ is nontrivial, since $\tilde{\pi} \otimes \chi \neq \tilde{\pi}$, which is clear by comparing the arithmetic conductors of $q_{\tilde{\pi} \otimes \chi}$ and $q_{\tilde{\pi}}$. If $\chi$ is trivial, then $L_{\chi}(s) = L_1(s)$.  We also define $\Lambda_{\chi}(s)\coloneqq \Lambda(s,\pi\times(\tilde{\pi}\otimes\chi))$ and $\Lambda_1(s)\coloneqq\Lambda(s,\pi\times\tilde{\pi})$.

\subsection{Preliminaries}

We define $a_{\pi}(m)$ and $a_{\pi\times\tilde{\pi}}(m)$ by the Dirichlet series identities
\begin{align*}
	 \sum_{n=1}^{\infty}\frac{a_{\pi}(m)\Lambda(m)}{m^s}&=-\frac{L'}{L}(s,\pi),\qquad 
	  \sum_{n=1}^{\infty}\frac{a_{\pi\times\tilde{\pi}}(m)\Lambda(m)}{m^s}=-\frac{L'}{L}(s,\pi\times\tilde{\pi}),\qquad\re(s)>1.
\end{align*}
The local calculations in \cite[Lemma 2.1]{LRS_selberg} show that if $\chi\pmod{q}$ is a primitive Dirichlet character and $\gcd(q,q_{\pi})=1$, then for all primes $p$, we have that
\begin{equation}
	\label{eqn:LRS_decouple}
L(s,\pi_p\times(\tilde{\pi}\otimes\chi)_p) = \prod_{j=1}^n \prod_{j'=1}^{n}(1-\alpha_{j,j',\pi\times\tilde{\pi}}(p)\chi(p)p^{-s})^{-1}.
\end{equation}
It follows from \eqref{eqn:separate_dirichlet_coeffs} and \eqref{eqn:LRS_decouple} that if $\gcd(m,q_{\pi})=1$, then $a_{\pi\times(\tilde{\pi}\otimes\chi)}(m) = |a_{\pi}(m)|^2 \chi(m)$.  We now provide a convenient expression for $-L_{\chi}'(s)/L_{\chi}(s)$.

\begin{lemma}
\label{lem:convenient_factorization}
Let $\pi\in\mathfrak{F}_n$.  Let $\psi\pmod{q}$ be a Dirichlet character with $q$ such that $\gcd(q,q_{\pi})=1$ and $\chi$ be the primitive Dirichlet character that induces $\psi$.  Let $\delta(\chi)=1$ if $\chi$ is trivial and $\delta(\chi)=0$ otherwise.  There exists a function $H_{\pi}(s;\chi,\psi)$ such that in the region $\re(s)\geq 1-\frac{1}{n^2}$,
	\begin{enumerate}
		\item $H_{\pi}(s;\chi,\psi)$ is holomorphic,
		\item $|H_{\pi}(s;\chi,\psi)|\ll_{\pi} \log(q(3+|\im(s)|))$, and
		\item we have the identity
		\[
	\sum_{m=1}^{\infty}\frac{|a_{\pi}(m)|^2\psi(m)\Lambda(m)}{m^s}=\frac{\delta(\chi)}{s-1}-\frac{\Lambda_{\chi}'}{\Lambda_{\chi}}(s) +H_{\pi}(s;\chi,\psi).
	\]
	\end{enumerate}
\end{lemma}
\begin{proof}
Suppose first that $\psi$ is primitive, in which case $\psi=\chi$ and the function $H_{\pi}(s;\chi,\psi)$ is
\begin{multline*}
\frac{\log q_{\pi\times(\tilde{\pi}\otimes\chi)}}{2}+\frac{L'}{L}(s,\pi_{\infty}\times(\tilde{\pi}\otimes\chi)_{\infty})+\frac{\delta(\chi)}{s}+\sum_{p \mid q_{\pi}q}\frac{L'}{L}(s,\pi_{p}\times(\tilde{\pi}'\otimes\chi)_p)\\
-\sum_{p\mid q_{\pi}q}\sum_{\substack{1\leq j\leq n \\ 1\leq j'\leq n}}\frac{\alpha_{j,\pi}(p)\overline{\alpha_{j',\pi'}(p)}\chi(p)\log p}{p^s-\alpha_{j,\pi}(p)\overline{\alpha_{j',\pi}(p)}\chi(p)}.
\end{multline*}
This is holomorphic and bounded as claimed for $\re(s)\geq 1-\frac{1}{n^2}$ by \eqref{eqn:LRS_2}, \eqref{eqn:BH}, and Stirling's formula.  If $\psi$ is not primitive and $\chi$ is the primitive character that induces $\psi$, then in the same region, we have
\[
	\Big|\sum_{m=1}^{\infty}\frac{|a_{\pi}(m)|^2\psi(m)\Lambda(m)}{m^s}-\sum_{m=1}^{\infty}\frac{|a_{\pi}(m)|^2\chi(m)\Lambda(m)}{m^s}\Big|\leq\sum_{p \mid q}\sum_{k\geq 1}\frac{|a_{\pi}(p^k)|^2\log p}{p^{k\mathop{\re}(s)}},
\]
which is bounded as desired using \eqref{eqn:LRS_2} again.
\end{proof}

We will the following zero-free region for $L_{\chi}(s)$ and Siegel-type bound on any real exceptional zeros.  We will prove this theorem later.
\begin{theorem}
\label{thm:ZFR}
	Let $Q\geq 3$ and $\pi\in\mathfrak{F}_n$.  There exists an effectively computable constant $\Cl[abcon]{ZFR_BV_final}=\Cr{ZFR_BV_final}(\pi)>0$ such that for all primitive Dirichlet characters $\chi\pmod{q}$ with $q\leq Q$ and $\gcd(q,q_{\pi})=1$ with at most one exception, the $L$-function $L(s,\pi\times(\tilde{\pi}\otimes\chi))$ does not vanish in the region
	\[
	\re(s)\geq 1-\frac{\Cr{ZFR_BV_final}}{\log(Q(|\im(s)|+3))}.
	\]
	If the exceptional character $\chi_1\pmod{q_1}$ exists, then $L(s,\pi\times(\tilde{\pi}\otimes\chi_1))$ has at most one zero, say $\beta_1$, in this region; $\beta_1$ is both real and simple; and $\chi_1$ must be quadratic.  Moreover, for all $\epsilon>0$, there exists an ineffective constant $c_{\pi}(\epsilon)>0$ such that $\beta_1\leq 1-c_{\pi}(\epsilon)q_1^{-\epsilon}$.
\end{theorem}

\subsection{Proof of \texorpdfstring{\cref{thm:BV}}{Theorem \ref*{thm:BV}}}

We follow Gallagher's proof of the Bombieri--Vinogradov theorem in \cite{Gallagher_BV}, with $n\geq 2$ and $\pi\in\mathfrak{F}_n$ fixed at the onset.  Note that the function
\[
\psi_k(y;q,a)\coloneqq \frac{1}{k!}\sum_{\substack{m\leq y\\ m\equiv a\pmod{q}}}|a_{\pi}(m)|^2\Lambda(m)\Big(\log\frac{y}{m}\Big)^k
\]
is monotonically increasing as a function of $y$ for each $k\geq 0$.  Thus, if $0<\lambda\leq 1$, then
\[
\frac{1}{\lambda}\int_{e^{-\lambda}y}^{y}\psi_{k-1}(t;q,a)\frac{dt}{t}\leq \psi_{k-1}(y;q,a)\leq \frac{1}{\lambda}\int_{y}^{e^{\lambda}y}\psi_{k-1}(t;q,a)\frac{dt}{t}.
\]
The integrals, which equal $\psi_k(y;q,a)-\psi_k(e^{-\lambda}y;q,a)$ and $\psi_k(e^{\lambda}y;q,a)-\psi_k(y;q,a)$ respectively, both have the same asymptotic expansion, namely
\[
(\lambda+O(\lambda^2))\frac{y}{\varphi(q)}+O(\max_{y\leq ex}|r_k(y;q,a)|),\qquad r_k(y;q,a)\coloneqq \psi_k(y;q,a)-\frac{y}{\varphi(q)},
\]
where $\varphi$ is Euler's totient function.  Thus, we have the bounds
\[
\max_{y\leq x}|r_{k-1}(y;q,a)|\ll \frac{\lambda x}{\varphi(q)}+\frac{1}{\lambda}\max_{y\leq ex}|r_k(y;q,a)|
\]
and
\[
\sum_{\substack{q\leq x^{\theta} \\ \gcd(q,q_{\pi})=1}}\max_{\gcd(a,q)=1}\max_{y\leq x}|r_{k-1}(y;q,a)|\ll \lambda x\log x+\frac{1}{\lambda}\sum_{\substack{q\leq x^{\theta} \\ \gcd(q,q_{\pi})=1}}\max_{\gcd(a,q)=1}\max_{y\leq ex}|r_k(y;q,a)|.
\]
It follows by induction on $k$ that
\[
\sum_{\substack{q\leq x^{\theta} \\ \gcd(q,q_{\pi})=1}}\max_{\gcd(a,q)=1}\max_{y\leq x}|r_0(y;q,a)|\ll_k \lambda x\log x + \frac{1}{\lambda^{2^k-1}}\sum_{\substack{q\leq x^{\theta} \\ \gcd(q,q_{\pi})=1}}\max_{\gcd(a,q)=1}\max_{y\leq ex}|r_k(y;q,a)|.
\]

\begin{proposition}
	\label{prop:BV}
	If $\theta<1/(9n^3)$ is fixed and $k=9n^2+1$, then for any $B>0$, we have
	\[
	\sum_{\substack{q\leq x^{\theta} \\ \gcd(q,q_{\pi})=1}}\max_{\gcd(a,q)=1}\max_{y\leq ex}|r_k(y;q,a)|\ll_{\pi,B} \frac{x}{(\log x)^B}.
	\]
\end{proposition}

\begin{proof}[\cref{prop:BV} implies \cref{thm:BV}]
	It follows from \cref{prop:BV} that
	\[
	\sum_{\substack{q\leq x^{\theta} \\ \gcd(q,q_{\pi})=1}}\max_{\gcd(a,q)=1}\max_{y\leq x}|r_0(y;q,a)|\ll_{\pi,B} \lambda x\log x+\frac{1}{\lambda^{2^k-1}}\frac{x}{(\log x)^B}.
	\]
	To finish, we choose $B=2^k(A+1)-1$ and $\lambda = (\log x)^{-(B+1)/2^k}$.
\end{proof}

\subsection{Proof of \texorpdfstring{\cref{prop:BV}}{Proposition \ref*{prop:BV}}}

Let $\pi\in\mathfrak{F}_n$, $k=9n^2+1$, $Q= x^{\theta}$ for some fixed $0<\theta<1/(9n^3)$, and $q\leq Q$.  We have the decomposition
\[
\frac{1}{k!}\sum_{\substack{m\leq y \\ m\equiv a\pmod{q}}}|a_{\pi}(m)|^2\Lambda(m)\Big(\log\frac{y}{m}\Big)^k =\frac{1}{k!}\sum_{\psi\pmod{q}}\frac{\bar{\psi}(a)}{\varphi(q)}\sum_{\substack{y\leq m}}|a_{\pi}(m)|^2\Lambda(m)\psi(n)\Big(\log\frac{y}{m}\Big)^k.
\]
By \cref{lem:convenient_factorization} and Mellin inversion, this equals
\[
\frac{1}{\varphi(q)}\sum_{\psi\pmod{q}}\sum_{\substack{\textup{$\chi$ primitive} \\ \textup{$\chi$ induces $\psi$}}}\frac{\bar{\chi}(a)}{2\pi i}\int_{3-i\infty}^{3+i\infty}\Big(-\frac{\Lambda_{\chi}'}{\Lambda_{\chi}}(s)+H_{\pi}(s;\chi,\psi)\Big)\frac{y^s}{s^{k+1}}ds.
\]
Since a primitive character $\chi\pmod{q}$ induces characters to moduli that are a multiple of $q$, it follows from the bound $\sum_{f\leq Q,~q|f}\varphi(f)^{-1}\ll\frac{\log Q}{\varphi(q)}$ that
\begin{equation}
	\label{eqn:reduc_1}
	\begin{aligned}
		&\sum_{\substack{q\leq Q \\ \gcd(q,q_{\pi})=1}}~\max_{\gcd(a,q)=1}~\max_{y\leq x}\Big|\frac{1}{k!}\sum_{\substack{m\leq y\\ m\equiv a\pmod{q}}}|a_{\pi}(m)|^2\Lambda(m)\Big(\log\frac{y}{m}\Big)^k- \frac{y}{\varphi(q)}\Big|\\
&\ll  \sum_{\substack{q\leq Q \\ \gcd(q,q_{\pi})=1}}\frac{\log Q}{\varphi(q)}~\sum_{\substack{\chi\pmod{q} \\ \textup{$\chi$ primitive}}}\max_{y\leq x}\Big|\frac{1}{2\pi i}\int_{3-i\infty}^{3+i\infty}\Big(-\frac{\Lambda_{\chi}'}{\Lambda_{\chi}}(s)+H_{\pi}(s;\chi,\chi)\Big)\frac{y^s}{s^{k+1}}ds\Big|.
	\end{aligned}
\end{equation}

Observe that by \cref{lem:convenient_factorization} and our range of $\theta$, \eqref{eqn:reduc_1} equals
\begin{equation}
\label{eqn:reduc_11}
\begin{aligned}
	& \sum_{\substack{q\leq Q \\ \gcd(q,q_{\pi})=1}}\frac{\log Q}{\varphi(q)}~\sum_{\substack{\chi\pmod{q} \\ \textup{$\chi$ primitive}}}\max_{y\leq x}\Big|-\sum_{\substack{L_{\chi}(\rho)=0 }}\frac{y^{\rho}}{\rho^{k+1}}+\frac{1}{2\pi i}\int_{1-\frac{1}{n^2}-i\infty}^{1-\frac{1}{n^2}+i\infty}H_{\pi}(s;\chi,\chi)\frac{y^s}{s^{k+1}}ds\Big|\\
	&\ll_{\pi,B}\log Q \sum_{\substack{q\leq Q \\ \gcd(q,q_{\pi})=1}}\frac{1}{\varphi(q)}~\sum_{\substack{\chi\pmod{q} \\ \textup{$\chi$ primitive}}}\sum_{\substack{L_{\chi}(\rho)=0 }}\frac{y^{\beta}}{|\rho|^{k+1}}+\frac{x}{(\log x)^B},
\end{aligned}
\end{equation}
where $\rho=\beta+i\gamma$ ranges over the nontrivial zeros of $L_{\chi}(s)$.  In light of the bounds $\frac{1}{q} \leq \frac{1}{\varphi(q)}\ll \frac{\log(eq)}{q}$, we dyadically decompose $[1,Q]$ into $O(\log Q)$ subintervals and find that \eqref{eqn:reduc_11} is
\begin{equation}
\label{eqn:reduc_2}
\begin{aligned}
&\ll_{\pi,B} x(\log Q)^2\sup_{3\leq R\leq Q}\sum_{\substack{q\leq R \\ \gcd(q,q_{\pi})=1}}\frac{1}{\varphi(q)}~\sum_{\substack{\chi\pmod{q} \\ \textup{$\chi$ primitive}}}~\sum_{\substack{\rho=\beta+i\gamma  }}\frac{x^{\beta-1}}{|\rho|^k}+\frac{x}{(\log x)^B}\\
&\ll_{\pi,B} x(\log Q)^3\sup_{3\leq R\leq x^{\theta}}\frac{1}{R}\sum_{\substack{q\leq R \\ \gcd(q,q_{\pi})=1}}~\sum_{\substack{\chi\pmod{q} \\ \textup{$\chi$ primitive}}}~\sum_{\substack{\rho=\beta+i\gamma  }}\frac{x^{\beta-1}}{|\rho|^k}+\frac{x}{(\log x)^B}\\
&\ll_{\pi,B} x(\log x)^3\sup_{3\leq R\leq x^{\theta}}\frac{1}{R}\Big(\frac{x^{\beta_1-1}}{\beta_1^k}+\sum_{\substack{q\leq R \\ \gcd(q,q_{\pi})=1}}~\sum_{\substack{\chi\pmod{q} \\ \textup{$\chi$ primitive}}} \sum_{\substack{\rho=\beta+i\gamma\neq\beta_1  }}\frac{x^{\beta-1}}{|\rho|^k}\Big)+\frac{x}{(\log x)^B}.
\end{aligned}	
\end{equation}
where $\beta_1$ is the exceptional zero in \cref{thm:ZFR}.  The term $x^{\beta_1-1}/\beta_1^k$ is omitted if $\beta_1$ does not exist.

If $\beta_1$ exists as in \cref{thm:ZFR} and the supremum is achieved when $R\leq (\log x)^{4B}$, then we apply \cref{thm:ZFR} with $\epsilon=\frac{1}{8B}$ and conclude that the contribution from such a zero is absorbed in our error term.  If the supremum is achieved when $R>(\log x)^{4B}$, then contribution from $\beta_1$ is trivially absorbed in our error term.  Hence \eqref{eqn:reduc_2} is
\begin{equation}
\label{eqn:reduc_3}
\ll_{\pi,B} x(\log x)^3\sup_{3\leq R\leq x^{\theta}}\frac{1}{R} \sum_{\substack{q\leq R \\ \gcd(q,q_{\pi})=1}}~\sum_{\substack{\chi\pmod{q} \\ \textup{$\chi$ primitive}}}\sum_{\substack{\rho=\beta+i\gamma\neq\beta_1 }}\frac{x^{\beta-1}}{|\rho|^k}+\frac{x}{(\log x)^B}.	
\end{equation}

Let us now consider the zeros $\rho $ with $|\rho |<\frac{1}{4}$.  The number of such zeros is $\ll R^2\log R$.  From the consideration of the corresponding zeros $1-\rho$ of $L_{\bar{\chi}}(s)$, we deduce that $|\rho |\gg x^{-1/(4k)}$.  Thus, the contribution from these zeros is $\ll R x^{\frac{1}{4}+\frac{k}{4k}}\log R\ll Q x^{\frac{1}{2}}\log Q \ll_{B} x(\log x)^{-B}$.  Define $T_0=0$ and $T_j=2^{j-1}$ for $j\geq 1$.  The above discussion shows that \eqref{eqn:reduc_3} is
\begin{equation}
	\label{eqn:reduc_4}
	\begin{aligned}
	&\ll_{\pi,B} x(\log x)^3\sup_{3\leq R\leq x^{\theta}}\frac{1}{R}\sum_{\substack{q\leq R \\ \gcd(q,q_{\pi})=1}}~\sum_{\substack{\chi\pmod{q} \\ \textup{$\chi$ primitive}}} \sum_{\substack{\rho=\beta+i\gamma\neq \beta_1  \\ |\rho|>\frac{1}{4}}}\frac{x^{\beta-1}}{|\rho|^k}+\frac{x}{(\log x)^B}\\
	&\ll_{\pi,B} x(\log x)^3\sup_{3\leq R\leq x^{\theta}}\frac{1}{R}\sum_{j=1}^{\infty}~\sum_{\substack{q\leq R \\ \gcd(q,q_{\pi})=1}}~\sum_{\substack{\chi\pmod{q} \\ \textup{$\chi$ primitive}}}~\sideset{}{'}\sum_{\substack{\rho=\beta+i\gamma\neq\beta_1  \\ |\rho|\geq\frac{1}{4} \\ T_{j-1}\leq|\gamma|\leq T_j}}\frac{x^{\beta-1}}{|\rho|^k}+\frac{x}{(\log x)^B}.
	\end{aligned}
\end{equation}

If $|\rho|\geq \frac{1}{4}$ and $T_{j-1}\leq|\gamma|\leq T_j$, then $|\rho|\geq \max\{T_{j-1},\frac{1}{4}\}\geq T_j/4$ and $|\rho|\gg |\gamma|+3$.  Therefore, if $\delta = \min\{1-9n^3\theta,\tfrac{1}{2}\}$, then
\[
x^{\beta-1}|\rho|^{-k}\ll T_j^{-\frac{1}{2}}(|\gamma|+1)^{-\frac{1}{2}}x^{-(1-\beta)\delta}(x^{1-\delta}T_j^{k-1})^{-(1-\beta)}\ll T_j^{-\frac{1}{2}}(|\gamma|+1)^{-\frac{1}{2}}x^{-(1-\beta)\delta}(R^{\frac{1-\delta}{\theta}}T_j^{k-1})^{\beta-1}.
\]
Since $\rho=\beta+i\gamma\neq\beta_1$, it follows from \cref{thm:ZFR} that
\[
(|\gamma|+1)^{-\frac{1}{2}}x^{-\delta(1-\beta)}\leq e^{-\delta \eta_{\pi}(x,R)},\qquad \eta_{\pi}(x,R) \coloneqq  \inf_{t\geq 3}\Big[c_{\pi}\frac{\log x}{\log(Rt)}+\log t\Big],
\]
so \eqref{eqn:reduc_4} is
\[
\ll_{\pi,B}x(\log x)^3 \sup_{3\leq R\leq x^{\theta}}\frac{e^{-\delta \eta(x,R)}}{R}\sum_{j=1}^{\infty}T_j^{-\frac{1}{2}}\sum_{\substack{q\leq R \\ \gcd(q,q_{\pi})=1}}~\sum_{\substack{\chi\pmod{q} \\ \textup{$\chi$ primitive}}}\sum_{\substack{\rho=\beta+i\gamma\neq \beta_1\\ |\gamma|\leq T_j}}(R^{\frac{1-\delta}{\theta}}T_j^{k-1})^{\beta-1}+\frac{x}{(\log x)^B}.
\]

Define
\[
N_{\pi}^*(\sigma,T,R)=\sum_{\substack{q\leq R \\ \gcd(q,q_{\pi})=1}} \sum_{\substack{\chi\bmod q \\ \textup{$\chi$ primitive}}}|\{\rho=\beta+i\gamma\colon \beta\geq\sigma,~|\gamma|\leq T,~L_{\chi}(\rho)=0\}|.
\]
By \eqref{eqn:BH}, there exists an effectively computable constant $\Cl[abcon]{const_family_scale}=\Cr{const_family_scale}(n)>0$ such that
\[
\{\tilde{\pi}\otimes\chi\colon \textup{$ \chi\pmod{q_{\chi}}$ primitive, $\gcd(q_{\chi},q_{\pi})=1$, $q_{\chi}\leq Q$}\}\subseteq\mathfrak{F}_{n}(\Cr{const_family_scale}C(\pi)Q^n).
\]
Thus, by \eqref{eqn:ZDE_Q}, we have $N_{\pi}^*(\sigma,T,R)\ll_{\pi,\epsilon}(R^n T)^{9n^2(1-\sigma)+\epsilon}$. Partial summation yields
\begin{equation}
	\label{eqn:zerobound}
	\begin{aligned}
	&\sum_{\substack{q\leq R \\ \gcd(q,q_{\pi})=1}}~\sum_{\substack{\chi\pmod{q} \\ \textup{$\chi$ primitive}}}\sum_{\substack{\rho=\beta+i\gamma\neq \beta_1\\ |\gamma|\leq T_j}}(R^{\frac{1-\delta}{\theta}}T_j^{k-1})^{-(1-\beta)}\\
&\ll (R^{\frac{1-\delta}{\theta}}T_j^{k-1})^{-1}N_{\pi}^*(0,T_j,R)+\log(RT_j)\int_0^1 (R^{\frac{1-\delta}{\theta}}T_j^{k-1})^{-\sigma}N_{\pi}^*(1-\sigma,T_j,R)d\sigma\\
&\ll_{\pi,\epsilon} (R T_j)^{\epsilon}\Big((R^{\frac{1-\delta}{\theta}}T_j^{k-1})^{-1}R T_j+\int_0^1 (R^{\frac{1-\delta}{\theta}}T_j^{k-1})^{-\sigma}(R^n T_j)^{9n^2\sigma}d\sigma\Big).	
	\end{aligned}
\end{equation}
Our choices for $\theta$, $k$, and $\delta$ ensure that \eqref{eqn:zerobound} is $\ll_{\pi,\epsilon} (R T_j)^{\epsilon}$.  Thus, \eqref{eqn:reduc_4} is
\begin{equation}
	\label{eqn:reduc_6}
	\begin{aligned}
	&\ll_{\pi,B,\epsilon} x(\log x)^3 \sup_{3\leq R\leq x^{\theta}}e^{-\delta \eta_{\pi}(x,R)}\frac{R^{\epsilon}}{R}\sum_{j=1}^{\infty}T_j^{-\frac{1}{2}+\epsilon}+\frac{x}{(\log x)^B}\\
	&\ll_{\pi,B,\epsilon} x(\log x)^3 \sup_{3\leq R\leq x^{\theta}}e^{-\delta \eta_{\pi}(x,R)}\frac{R^{\epsilon}}{R}+\frac{x}{(\log x)^B}.
	\end{aligned}
\end{equation}
A small calculation (cf.\ \cite[Section 4]{TZ3}) shows that there is a constant $\Cl[abcon]{last_step_BV}=\Cr{last_step_BV}(\pi)>0$ such that $e^{-\delta\eta_{\pi}(x,R)}\ll_{\pi,B,\epsilon} \exp(-\Cr{last_step_BV}\frac{\log x}{\log R})+\exp(-\Cr{last_step_BV}\sqrt{\log x})$, and \cref{thm:BV} follows.

\subsection{Proof of \texorpdfstring{\cref{thm:ZFR}}{Theorem \ref*{thm:ZFR}}}

Although the proof of \cref{thm:ZFR} contains only standard techniques, such zero-free regions for $L_{\chi}(s)$ are new even in the case when $\chi$ is trivial.  The case when $\chi$ is trivial was only recently handled unconditionally in \cite{HT}.

\begin{lemma}\cite[Theorem 2.1(1)]{HT}
\label{lem:ZFR_HT}
	Let $\pi\in\mathfrak{F}_n$.  There exists an absolute and effectively computable constant $\Cl[abcon]{ZFR_HT}>0$ such that $L_{1}(s)$ has at most one zero, say $\beta_1$, in the region $\re(s)\geq 1-\Cr{ZFR_HT}/\log(C(\pi)^n(|\im(s)|+e)^{n^2})$.  If $\beta_1$ exists, then it must be real and simple, and there exists an absolute and effectively computable constant $\Cl[abcon]{Siegel_HT}$ such that $\beta_1\leq 1-C(\pi)^{-\Cr{Siegel_HT}n}$.
\end{lemma}

We prove the corresponding result for $L_{\chi}(s)$.  The ideas in \cite{HT} inform our approach here

\begin{lemma}
\label{lem:ZFR_twist}
Let $\pi\in\mathfrak{F}_{n}$. Let $\chi\pmod{q}$ be a nontrivial primitive Dirichlet character such that $\gcd(q,q_{\pi})=1$.  There exists an effectively computable constant $\Cl[abcon]{ZFR_BV_1}=\Cr{ZFR_BV_1}(n)>0$ such that $L_{\chi}(s)$ has at most one zero, say $\beta_1$, in the region
\begin{equation}
\label{eqn:ZFR_chi}
\re(s)\geq 1-\frac{\Cr{ZFR_BV_1}}{\log(qC(\pi)(|\im(s)|+3))}.
\end{equation}
If the exceptional zero $\beta_1$  exists, then it is real and simple, and $\chi$ is quadratic.
\end{lemma}
\begin{proof}
Let $\chi\pmod{q}$ be a primitive nontrivial Dirichlet character, let $\psi$ be the primitive character that induces $\chi^2$, and let $\beta+i\gamma$ be a nontrivial zero of $L_{\chi}(s)$.  Define $\Pi_{\chi}=\pi\boxplus \pi\otimes\chi|\cdot|^{it}\boxplus\pi\otimes\bar{\chi}|\cdot|^{-it}$, and define
	\begin{equation}
	\label{eqn:aux_siegel}
	D(s)=L(s,\Pi_{\chi}\times\widetilde{\Pi}_{\chi})=L_{1}(s)^3 L_{\chi}(s+i\gamma)^2 L_{\bar{\chi}}(s-i\gamma)^2 L_{\psi}(s+2i\gamma)L_{\bar{\psi}}(s-2i\gamma).
	\end{equation}
The factor $L_1(s)^3$ has a pole of order 3 at $s=1$, and the hypothesis that $\gcd(q,q_{\pi})=1$ ensures that $L_{\chi}(s+i\gamma)^2 L_{\bar{\chi}}(s-i\gamma)^2$ is entire.  If $\psi$ is complex, then $L_{\psi}(s+2i\gamma)L_{\bar{\psi}}(s-2i\gamma)$ is entire; otherwise, it has poles of order 1 at $s=1\pm 2i\gamma$.  The additional poles when $\psi$ is real require some additional casework when $\gamma$ is close to zero.  For notational compactness, let $\mathcal{Q}_{\gamma}=qC(\pi)(|\gamma|+3)$.  Note that $-\frac{D'}{D}(s)$ has nonnegative Dirichlet coefficients per \cite[Lemma a]{Hoffstein}.

The functional equation for $L_{\chi}(s)$ together with the fact that $L(s,\pi\times\tilde{\pi})$ is a self-dual $L$-function (even if $L(s,\pi)$ itself is not self-dual) implies that if $\rho$ is a zero of $L_{\chi}(s)$, then $\bar{\rho}$ is a zero of $L_{\bar{\chi}}(s)$.  Thus, we have that
	\begin{equation}
	\label{eqn:order_beta}
	\mathop{\mathrm{ord}}_{s=\beta}D(s) \geq 4.
	\end{equation}
	
	Let $\omega$ denote a nontrivial zero of $D(s)$, $\delta(\psi)=1$ if $\psi$ is trivial, and $\delta(\psi)=0$ otherwise.  We apply \cref{lem:GHL} and \eqref{eqn:BH} to \eqref{eqn:aux_siegel}, concluding that if $1<\sigma<2$, then
	\begin{equation}
	\label{eqn:startpointzfr_1}
	\sum_{\omega}\re\Big(\frac{1}{\sigma-\omega}\Big)<\frac{3}{\sigma-1}+\delta(\psi)\frac{2(\sigma-1)}{(\sigma-1)^2+4\gamma^2}+\Cl[abcon]{zfrproof}\log\mathcal{Q}_{\gamma},
	\end{equation}
	where $\Cr{zfrproof}=\Cr{zfrproof}(n)>1$ is a suitable implied constant.  Since $\beta<1$ is, by hypothesis, one of the zeros in the sum in \eqref{eqn:startpointzfr_1}, we have by \eqref{eqn:order_beta} and nonnegativity that
	\begin{equation}
	\label{eqn:startingpoint_zfr}
	\frac{4}{\sigma-\beta}< \frac{3}{\sigma-1}+\delta(\psi)\frac{2(\sigma-1)}{4\gamma^2+(\sigma-1)^2}+\Cr{zfrproof}\log\mathcal{Q}_{\gamma}.
	\end{equation}
		
\subsubsection*{Case 1: Either $\chi$ is real and $|\gamma|\geq \frac{1}{7c_{22}\log\mathcal{Q}_0}$, or $\chi$ is complex}
	
	If $\sigma = 1+\frac{1}{5\Cr{zfrproof}\log \mathcal{Q}_{\gamma}}$, then
	\[
	\delta(\psi)\frac{2(\sigma-1)}{4\gamma^2+(\sigma-1)^2}\leq \frac{490}{149}\Cr{zfrproof}\log\mathcal{Q}_{\gamma}.
	\]
	Thus, \eqref{eqn:startingpoint_zfr} becomes
	\[
	\frac{4}{1+\frac{1}{5\Cr{zfrproof}\log\mathcal{Q}_{\gamma}}-\beta}\leq \frac{2874}{149}\Cr{zfrproof}\log\mathcal{Q}_{\gamma}.
	\]
	Upon solving for $\beta$, we conclude that $\beta\leq 1-\frac{1}{136\Cr{zfrproof}\mathcal{Q}_{\gamma}}$.
	
	\subsubsection*{Case 2: $\chi$ is real and $\gamma=0$}
	We start at \eqref{eqn:startpointzfr_1} with $\delta(\psi)=1$, $\gamma=0$, and $\sigma = 1+\frac{1}{3\Cr{zfrproof}\log\mathcal{Q}_0}$.  If there are $N$ zeros $\beta$ (with multiplicity) of $D(s)$ such that $\beta\geq 1-\frac{1}{96\Cr{zfrproof}\log\mathcal{Q}_0}$, then by \eqref{eqn:startpointzfr_1},
	\[
	\frac{32}{11}\Cr{zfrproof}N\log\mathcal{Q}_0=\frac{N}{\sigma-(1-\frac{1}{96\Cr{zfrproof}\log\mathcal{Q}_0})}\leq \frac{5}{\sigma-1}+\Cr{zfrproof}\log \mathcal{Q}_0=16\Cr{zfrproof}\log\mathcal{Q}_0.
	\]
	It follows that (since $N$ is an integer) $N\leq\lfloor\frac{11}{2}\rfloor=5$.  By the bound \eqref{eqn:order_beta}, $L_{\chi}(s)$ has at most one real zero $\beta$, necessarily simple, satisfying $\beta\geq 1-\frac{1}{96\Cr{zfrproof}\log\mathcal{Q}_0}$.
	
	\subsubsection*{Case 3: $\chi$ is quadratic and $0<|\gamma|<\frac{1}{7c_{22}\log \mathcal{Q}_0}$}
	We apply \cref{lem:GHL} to $L_{1}(s)L_{\chi}(s)$.  The only singularity is a simple pole at $s=1$.  Both $L_{1}(s)$ and $L_{\chi}(s)$ are self-dual, so if $\beta+i\gamma$ is a nontrivial zero of $L_1(s)L_{\chi}(s)=0$, then so is $\beta-i\gamma$.  By \eqref{eqn:LRS_decouple}, the $m$-th Dirichlet coefficient of $-\frac{L_{1}'}{L_{1}}(s)-\frac{L_{\chi}'}{L_{\chi}}(s)$ is $(1+\chi(m))a_{\pi\times\tilde{\pi}}(m)\Lambda(m)\geq 0$, so \cref{lem:GHL} yields
	\begin{equation}
	\label{eqn:depeche}
	2\frac{\sigma-\beta}{(\sigma-\beta)^2+\gamma^2}\leq\frac{1}{\sigma-1}+\Cr{zfrproof}\log \mathcal{Q}_{\gamma}.
	\end{equation}
	If $\sigma = 1+\frac{1}{2\Cr{zfrproof}\log\mathcal{Q}_{\gamma}}$ and $0<|\gamma|<\frac{1}{7\Cr{zfrproof}\log\mathcal{Q}_{0}}$ in \eqref{eqn:depeche}, then $\beta\leq 1-\frac{1}{8\Cr{zfrproof}\log\mathcal{Q}_{\gamma}}$.
	\end{proof}
	
We now prove a Siegel-type bound for $\beta_1$ in \cref{lem:ZFR_twist} (if it exists) using the ideas of  Hoffstein and Lockhart \cite{HL}.  This is new for all $\pi\in\mathfrak{F}_n$ with $n\geq 3$.  We begin with an auxiliary calculation.  Let $\chi\pmod{q}$ and $\chi'\pmod{q'}$ be distinct nontrivial quadratic Dirichlet characters such that $\gcd(q'q,q_{\pi})=1$, and let $\psi$ be the primitive character that induces $\chi'\chi$ (whose conductor is necessarily coprime to $q_{\pi}$).  These coprimality restrictions ensure that $\pi\neq \pi\otimes\chi$, $\pi\neq\pi\otimes\chi'$, $\pi\neq\pi\otimes\psi$,  $q_{\pi\otimes\chi}=q_{\pi}q^n$, $q_{\pi\otimes\chi'}=q_{\pi}(q')^n$, and $q_{\pi\otimes\psi}=q_{\psi}q^n$.  We also have that $q_{\psi}|q_{\chi}q_{\chi'}$.  Let
\begin{equation}
\label{eqn:Pi_def_Siegel}
L(s,\Pi^{\star})=L_{1}(s)L_{\chi}(s)L_{\chi'}(s)L_{\psi}(s).
\end{equation}
By the above discussion, $L(s,\Pi^{\star})$ is holomorphic apart from a simple pole at $s=1$.  It follows from \eqref{eqn:LRS_decouple} that if $k\geq 1$, then the $p^k$-th Dirichlet coefficient of $\log L(s,\Pi^{\star})$ equals
\[
k^{-1}a_{\pi\times\tilde{\pi}}(p^k)(1+\chi(p^k)+\chi'(p^k)+\psi(p^k))\geq 0.
\]
The nonnegativity of $1+\chi(p^k)+\chi'(p^k)+\psi(p^k)$ follows from the fact that this sum is a Dirichlet coefficient of the Dedekind zeta function of a biquadratic field.  Upon exponentiating, we find that the $m$-th Dirichlet coefficient $\lambda_{\Pi^{\star}}(m)$ of $L(s,\Pi^{\star})$ is nonnegative.

\begin{lemma}
	\label{lem:quad_nonneg}
	If $\pi\in\mathfrak{F}_n$ and $\chi$ is a primitive nontrivial real Dirichlet character, then $L_{\chi}(1)>0$ and the Dirichlet coefficients of $L_{1}(s)L_{\chi}(s)$ are nonnegative.
\end{lemma}
\begin{proof}
	Let $K$ be the quadratic field associated to $\chi$.  If $\pi_{\mathrm{BC}}$ is the base change of $\pi$ to an automorphic representation of $\mathrm{GL}_n(\mathbb{A}_K)$, then $L(s,\pi_{\mathrm{BC}}\times\tilde{\pi}_{\mathrm{BC}})=L_{1}(s)L_{\chi}(s)$.  Since $L(s,\pi_{\mathrm{BC}}\times\tilde{\pi}_{\mathrm{BC}})$ is holomorphic on $\mathbb{C}-\{1\}$ apart from a simple pole at $s=1$, the same holds for $L_{1}(s)L_{\chi}(s)$.  Since $\gcd(q,q_{\pi})=1$ (hence $L_{\chi}(s)$ is entire), the residue of $L(s,\pi_{\mathrm{BC}}\times\tilde{\pi}_{\mathrm{BC}})$ at $s=1$, which is positive, equals $L_{\chi}(1)\mathrm{Res}_{s=1}L_{1}(s)$.  Since $\mathrm{Res}_{s=1}L_{1}(s)>0$, it follows that $L_{\chi}(1)>0$.  The Dirichlet coefficients of  $L_{1}(s)L_{\chi}(s)$ are nonnegative per Case 3 in the proof of \cref{lem:ZFR_twist}.
\end{proof}

Let $0<\epsilon<1$, and let $\beta\in(1-\epsilon,1)$.  By \eqref{eqn:preconvex} and the above discussion, we have
	\begin{equation}
	\label{eqn:half_bound}
	\begin{aligned}
	L(\tfrac{1}{2}+it,\Pi^{\star})&\ll_{\pi,\epsilon} (q'q)^{\frac{n^2}{2}+\epsilon}(3+|t|)^{n^2+\epsilon},\\
	\mathop{\mathrm{Res}}_{s=1-\beta}L(s+\beta,\Pi^{\star})x^s\Gamma(s)&\ll_{\pi,\epsilon}L_{\chi}(1)(q'q)^{\epsilon}(1-\beta)^{-1}x^{1-\beta}.
	\end{aligned}
	\end{equation}
	If $x\geq 3$, then since $\lambda_{\Pi^{\star}}(1)=1$, we use \eqref{eqn:half_bound} to deduce that
	\begin{equation}
	\label{eqn:aux_calcul}
	\begin{aligned}
	\frac{1}{e}&\leq \sum_{m=1}^{\infty}\frac{\lambda_{\Pi^{\star}}(m)}{m^{\beta}}e^{-\frac{m}{x}}\\
	&=\frac{1}{2\pi i}\int_{3-i\infty}^{3+i\infty}L(s+\beta,\Pi^{\star})x^s\Gamma(s)ds\\
	&=\mathop{\mathrm{Res}}_{s=1-\beta}L(s+\beta,\Pi^{\star})x^s\Gamma(s)+L(\beta,\Pi^{\star})+\frac{1}{2\pi i}\int_{\frac{1}{2}-\beta-i\infty}^{\frac{1}{2}-\beta+i\infty}L(s+\beta,\Pi^{\star})x^s\Gamma(s)ds\\
	&\ll_{\pi,\epsilon}L_{\chi}(1)(q'q)^{\epsilon}(1-\beta)^{-1}x^{1-\beta}+L(\beta,\Pi^{\star})+(q'q)^{\frac{n^2}{2}+\epsilon}x^{\frac{1}{2}-\beta}.
	\end{aligned}
	\end{equation}
	
\begin{proposition}
\label{prop:Siegel_pre}
		Recall the notation and hypotheses of \cref{lem:ZFR_twist}.  If $\beta_1$ exists for a primitive character $\chi\pmod{q}$ such that $\gcd(q,q_{\pi})=1$, then for all $\epsilon>0$, there exists an (ineffective) constant $c_{\pi}'(\epsilon)>0$ such that $L_{\chi}(1)\geq c_{\pi}'(\epsilon)q^{-\epsilon}$.
\end{proposition}
\begin{proof}
It suffices to take $0<\epsilon<1$.  Let $\chi\pmod{q}$ and $\chi'\pmod{q'}$ be primitive quadratic Dirichlet characters with, let $\psi$ be the primitive character that induces $\chi'\chi$, and recall the definition of $L(s,\Pi^{\star})$ from \eqref{eqn:Pi_def_Siegel}.  Our proof consists of two cases.

First, suppose that there exists no primitive quadratic Dirichlet character $\omega$ such that $L_{\omega}(s)$ does not vanish for $s\in(1-\frac{\epsilon}{2},1)$.  It then follows that there exists a constant $\Cl[abcon]{siegel_first}=\Cr{siegel_first}(\pi)>0$ such that $L_{1}(s)L_{\chi}(s)\neq 0$ in the interval $s\in(1-\Cr{siegel_first}/\log q,1)$.  Since the Dirichlet coefficients of $L_{1}(s)L_{\chi}(s)$ are nonnegative (using \cref{lem:quad_nonneg}) and the residue of $L_{1}(s)L_{\chi}(s)$ at $s=1$ is $L_{\chi}(1)\mathrm{Res}_{s=1}L_{1}(s)$, it follows from \cite[Proposition 1.1]{HL} and \eqref{eqn:preconvex} that $L_{\chi}(1)\mathop{\mathrm{Res}}_{s=1}L_{1}(s)\gg_{\pi}(\log q)^{-1}$.  Since each term the left-hand side is positive (using \cref{lem:quad_nonneg}), the desired result follows.

Second, suppose that there exists $\chi'\pmod{q'}$ and $\beta\in(1-\frac{\epsilon}{2},1)$ such that $L_{\chi'}(\beta)=0$.  We may assume that $q'$ is minimal, subject to this condition.  Now, let $\chi$ be arbitrary.  If $q<q'$, then the minimality of $q'$ ensures that $L_{\chi}(s)\neq 0$ for $s\in(1-\frac{\epsilon}{2},1)$, and the preceding case implies the desired result.  Suppose now that $q\geq q'$.  If $L_{\chi}(s)$ has no real zero within a distance of $\Cr{ZFR_BV_1}/\log(3q'q C(\pi))$ of $s=1$, then since $q\geq q'$, $L_{\chi}(s)$ has no real zero within a distance of $\frac{1}{2}\Cr{ZFR_BV_1}/\log(3qC(\pi))$ of $s=1$.  Again, the desired result follows by the preceding case.  Finally, suppose that $L_{\chi}(s)$ has a real zero within a distance of $\Cr{ZFR_BV_1}/\log(3q'q C(\pi))$ of $s=1$.  At this stage, we assume that $\chi\neq\chi'$.

It follows from analysis nearly identical to the second case in \cref{lem:ZFR_twist} (with $L(s,\Pi^{\star})$ replacing $D(s)$) that $L(s,\Pi^{\star})$ has at most one real zero within distance $\Cr{ZFR_BV_1}/\log(3q'q C(\pi))$ from 1. Since we have supposed that $L_{\chi}(s)$ has a real zero within distance $\Cr{ZFR_BV_1}/\log(3q'q C(\pi))$ of $s=1$, the above discussion indicates that this is the sole real zero for $L(s,\Pi^{\star})$ within a distance of $\Cr{ZFR_BV_1}/\log(3q'q C(\pi))$ of $s=1$.  It follows that the zero $\beta$ of $L_{\chi'}(s)$ must satisfy
\[
\beta\leq 1-\frac{\Cr{ZFR_BV_1}}{\log(3q'qC(\pi))},\qquad \beta\in\Big(1-\frac{\epsilon}{2},1\Big).
\]
Since $L_{\chi'}(\beta)=0$, it follows that $L(\beta,\Pi^{\star})=0$.  Using \eqref{eqn:aux_calcul}, the above bounds on $\beta$, and the fact that $q\geq q'$, we find that
\[
1\ll_{\pi,\epsilon}L_{\chi}(1)q^{2\epsilon}x^{1-\beta}+q^{n^2+2\epsilon}x^{\frac{1}{2}-\beta}\ll_{\pi,\chi',\epsilon}L_{\chi}(1)q^{2\epsilon}x^{\frac{\epsilon}{2}}+q^{n^2+2\epsilon}x^{\frac{\epsilon}{2}-\frac{1}{2}}
\]
Choosing $x=q^{2n^2}/L_{\chi}(1)^2$ (which is at least 3 by \eqref{eqn:BH}  and \eqref{eqn:preconvex}) and solving for $L_{\chi}(1)$, we find that for all $\chi\neq\chi'$ and all $0<\epsilon<1$, there exists a constant $d_{\pi}(\epsilon)>0$ such that $L_{\chi}(1)\geq d_{\pi}(\epsilon)q^{-\epsilon(n^2+2)/(1-\epsilon)}$.  Upon rescaling $\epsilon$ to $\epsilon/(n^2+2+\epsilon)$, we have $L_{\chi}(1)\geq d_{\pi}(\epsilon/(n^2+2+\epsilon))q^{-\epsilon}$.  As long as $\chi\neq \chi'$, the constant $d_{\pi}(\epsilon/(n^2+2+\epsilon))$ is effective.  Once we decrease $d_{\pi}(\epsilon/(n^2+2+\epsilon))$ suitably to account for the case where $\chi=\chi'$, the claimed result holds for arbitrary $\chi$.
\end{proof}

	\begin{corollary}
	\label{cor:Siegel}
		Recall the notation and hypotheses of \cref{lem:ZFR_twist}.  If $\beta_1$ exists, then for all $\epsilon>0$, there exists an (ineffective) constant $c_{\pi}(\epsilon)>0$ such that $\beta_1\leq 1-c_{\pi}(\epsilon)q^{-\epsilon}$.
	\end{corollary}
	\begin{proof}
	If $\beta_1$ exists, then there exists $\sigma\in[\beta_1,1]$ such that $L_{\chi}'(\sigma)(1-\beta_1)=L_{\chi}(1)\geq c_{\pi}'(\frac{\epsilon}{2})q^{-\frac{\epsilon}{2}}$ by \cref{prop:Siegel_pre} and the mean value theorem.  The result follows once we establish the bound $L_{\chi}'(\sigma)\ll_{\pi,\epsilon} q^{\frac{\epsilon}{2}}$ for $\sigma\in[1-b_n/\log(3qC(\pi)),1]$, where $b_n>0$ is a suitable constant depending at most on $n$.  To prove this, we observe via Cauchy's integral formula that
\[
L_{\chi}'(1)=\frac{1}{2\pi i}\int_{|z-1|=\frac{1}{\log q}}\frac{L_{\chi}(z)}{(z-1)^2}dz\ll (\log q) \max_{|\xi-1|\leq\frac{1}{\log q}}|L_{\chi}(\xi)|,
\]	
in which case the desired bound follows from \eqref{eqn:preconvex}.
\end{proof}

We will show that among the primitive characters $\chi\pmod{q}$ with $q\leq Q$, we encounter very few with the property that $L_{\chi}(s)$ has an exceptional zero.

\begin{lemma}
	\label{lem:page}
	Let $Q\geq 3$.  There exists an effectively computable constant $\Cl[abcon]{ZFR_BV_3}=\Cr{ZFR_BV_3}(n)>0$ such that there is at most one real nontrivial primitive Dirichlet character $\chi_1\pmod{q_1}$ with $q_1\leq Q$ such that $L_{\chi_1}(s)$ has a real zero $\beta_1$ satisfying $\beta_1>1-\Cr{ZFR_BV_3}/\log(C(\pi)Q)$.
\end{lemma}
\begin{proof}
	Suppose to the contrary that $\chi\pmod{q}$ and $\chi'\pmod{q'}$ are two distinct such characters with $q,q'\leq Q$.  Let $\Pi = \pi\boxplus\pi\otimes\chi\boxplus\pi\otimes\chi'$, let $\psi$ be the primitive character that induces $\chi'\chi$, and let
	\[
	F(s) = L(s,\Pi\times\tilde{\Pi}) = L_{1}(s)^3L_{\chi}(s)^2 L_{\chi'}(s)^2 L_{\psi}(s)^2.
	\]
	By \cite[Lemma a]{Hoffstein}, the Dirichlet coefficients of $-\frac{F'}{F}(s)$ are nonnegative.  By \cref{lem:GHL}, there exists a constant $\Cl[abcon]{stirling}=\Cr{stirling}(n)\geq 1$ such that if $\omega$ runs through the nontrivial zeros of $F(s)$ and $1<\sigma<2$, then
	\begin{equation}
	\label{eqn:last_eqn}
	\sum_{\omega}\re\Big(\frac{1}{\sigma-\omega}\Big)<\frac{3}{\sigma-1}+\Cr{stirling}\log(C(\pi)Q).
	\end{equation}
	If $\sigma=1+\frac{1}{2\Cr{stirling}\log(C(\pi)Q)}$ and $M$ is the number (necessarily an integer) of real zeros (counting multiplicity) of $F(s)$ that are at least $1-\frac{1}{14\Cr{stirling}\log(C(\pi)Q)}$, then it follows from \eqref{eqn:last_eqn} that
	\[
	\frac{M}{1+\frac{1}{2\Cr{stirling}\log(C(\pi)Q)}-(1-\frac{1}{14\Cr{stirling}\log(C(\pi)Q)})}<\frac{3}{(1+\frac{1}{2\Cr{stirling}\log(C(\pi)Q)})-1}+\Cr{stirling}\log(C(\pi)Q).
	\]
	This implies that $M\leq 3$.  But if $L_{\chi}(s)$ and $L_{\chi'}(s)$ both have real zeros that are larger than $1-\frac{1}{14\Cr{stirling}\log(C(\pi)Q)}$, then $F(s)$ has 4 such zeros, a contradiction.  The lemma now follows.
\end{proof}

\begin{proof}[Proof of \cref{thm:ZFR}]
	This follows from \cref{cor:Siegel} and \cref{lem:ZFR_HT,lem:ZFR_twist,lem:page}.
\end{proof}

\bibliographystyle{abbrv}
\bibliography{Humphries_Thorner_GLmxGLn_zero_density}

\end{document}